\documentclass{amsart}
\usepackage[utf8]{inputenc}

\usepackage{amsfonts, amsthm, amsmath, enumerate, amssymb,mathrsfs, tikz, tikz-cd, ifthen, color, chngcntr, mathtools, multicol,array, bbm, euscript}

\usetikzlibrary{arrows}
\usetikzlibrary{decorations.markings}

\newtheorem{theorem}[subsection]{Theorem}

\newtheorem{proposition}[subsection]{Proposition}
\newtheorem{lemma}[subsection]{Lemma}
\newtheorem{remark}[subsection]{Remark}
\newtheorem{corollary}[subsection]{Corollary}
\theoremstyle{definition}
\newtheorem{definition}[subsection]{Definition}
\newtheorem{example}[subsection]{Example}

\definecolor{light-gray}{gray}{.65}

\newcommand{\R}{\mathbb{R}}
\newcommand{\Z}{\mathbb{Z}}

\newcommand{\Hex}{\operatorname{Poly}}
\newcommand{\Sph}{\operatorname{Sph}}

\newcommand{\cthree}[3]{
	\draw[thick, #3] (#1+1/2, -5) -- (#1+1/2, 5);
	\draw[thick, #3] (#1+#2+1/2, -5) -- (#1+#2+1/2, 5);
	{\ifthenelse{#2>1}{
	\foreach \y in {-5,...,2} {
	\draw[thick, #3] (#1+1/2, \y +1/2) -- (#1+#2+1/2, \y + #2+1/2);
	}
	}{
	\foreach \y in {-5,-4,-3,-2,-1,0,1,2,3} {\draw[thick, #3] (#1+1/2, \y +1/2) -- (#1+#2+1/2, \y + #2+1/2);}}
	};
}

\title{An equivariant surgery classification of $C_p$-surfaces}
\author{Kelly Pohland}

\begin{document}

\maketitle

\begin{abstract}
    Let $p$ be an odd prime, and let $C_p$ denote the cyclic group of order $p$. We use equivariant surgery methods to classify all closed, connected $2$-manifolds with an action of $C_p$. We additionally provide a way to construct representatives of each isomorphism class using a series of equivariant surgery operations. The results in this paper serve as an odd prime analogue to a similar classification proved by Dan Dugger. 
\end{abstract}

\section{Introduction}

Let $p$ be an odd prime, and let $C_p$ denote the cyclic group of order $p$. In this paper, we classify all closed and connected $2$-manifolds with an action of $C_p$ up to equivariant isomorphism. More specifically, we define ways of constructing classes of $C_p$-surfaces using equivariant surgery methods and prove that all $C_p$-surfaces can be constructed in this way. 

Dugger gave a similar classification of $C_2$-surfaces in \cite{Dug19}. In his paper, Dugger gave a complete list of isomorphism classes of $C_2$-surfaces and developed a full set of invariants which determine the isomorphism class of a given surface with involution. We use similar methods to show that all nontrivial, closed, connected $C_p$-surfaces are in one of six families of isomorphism classes of $C_p$-surfaces. Various papers have treated aspects of the classification result in Theorems \ref{orientableclassification} and \ref{nonorientableclassification}, mostly focusing on the orientable case \cite{Asoh, Bro91, BCN, Nielsen, Sch29, Smith}. Previous treatments of this classification problem give particular interest to using invariants to quantify the number of isomorphism types of equivariant surfaces. The new idea presented in this classification and in that of \cite{Dug19} is the construction of isomorphism classes via equivariant surgeries.

By giving a geometric construction of the surfaces, we provide additional information which allows us to use the classification in a new way. One such application is in the computation of $RO(G)$-graded Bredon cohomology, an important algebraic invariant in equivariant homotopy theory. The decomposition into surgery pieces informs the construction of cofiber sequences which give rise to long exact sequences on cohomology. Hazel used Dugger's classification to compute the cohomology of $C_2$-surfaces in this Bredon theory \cite{Haz19a}, and this author performed similar computations in the $p=3$ case using the classification presented in this paper \cite{Pohl}. 

The idea behind our classification result is to show that all $C_p$-surfaces can be described in terms of other simpler $C_p$-surfaces. Some examples of these ``building block'' surfaces are $S^{2,1}$ and $M_1^{\text{free}}$ which can be described as the $2$-sphere and torus (respectively) rotating about the axis passing through each of their centers. Other examples include the non-orientable spaces $N_2^{\text{free}}$ and $N_1[1]$ whose $C_p$-actions are shown in Figure \ref{examples1} in the case $p=5$. The final family of spaces needed for our classification is denoted $\Hex_n$ for $n\geq 1$. We can think of $\Hex_1$ as a $2p$-gon with opposite edges identified and a rotation action of $e^{2\pi i/p}$. Then $\Hex_n$ consists of $n$ copies of $\Hex_1$ glued together in a particular way. These surfaces are described in greater detail in Section \ref{surgintro}, but  we can also see this gluing demonstrated in Figure \ref{hexagonparty} in the case $p=3$. 

\begin{figure}
\begin{center}
\includegraphics[scale=.4]{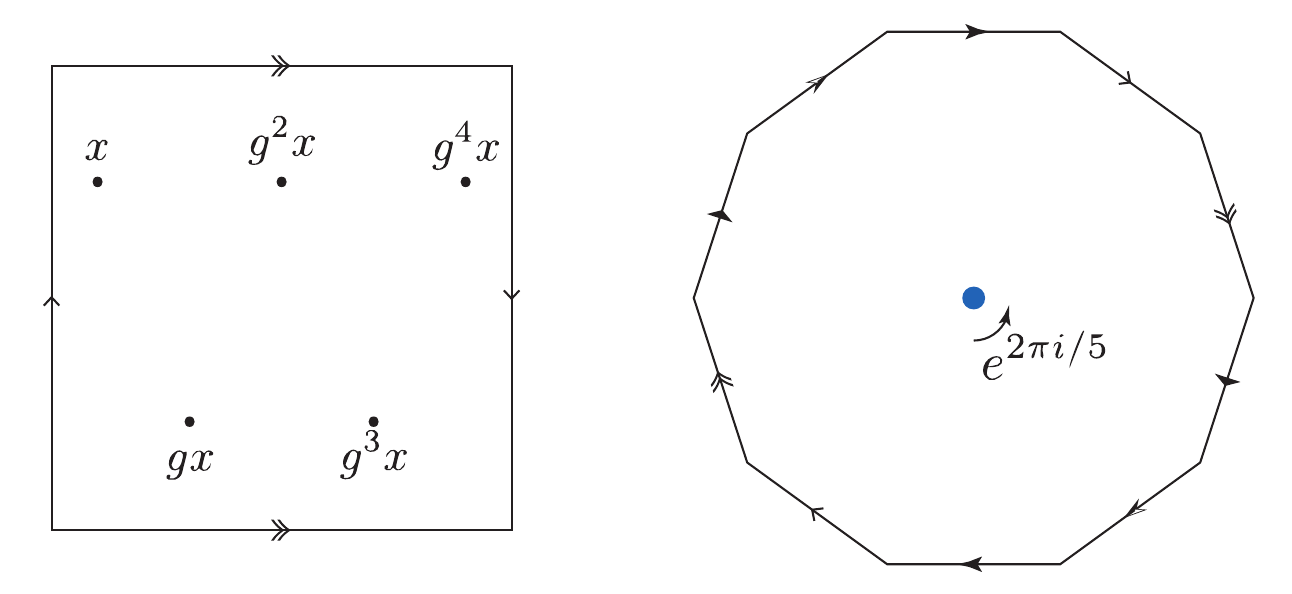}
	\end{center}
\caption{\label{examples1} The spaces $N_2^{\text{free}}$ (left) and $N_1[1]$ (right) in the case $p=5$.}
\end{figure}

\begin{figure}
\begin{center}
\includegraphics[scale=.5]{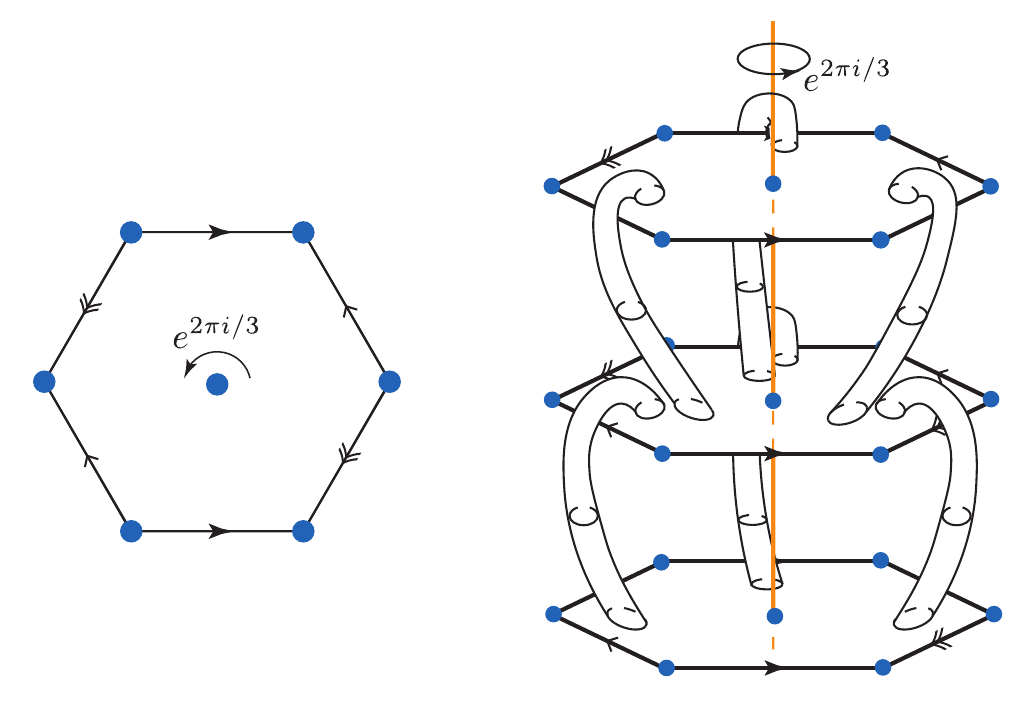}
	\end{center}
\caption{\label{hexagonparty} The spaces $\Hex_1$ (left) and $\Hex_3$ (right) in the case $p=3$.}
\end{figure}

Before precisely stating the classification result, let us introduce some equivariant surgery operations. 

Let $Y$ be a non-equivariant surface and $X$ a non-trivial $C_p$-surface. We construct a new $C_p$-surface by removing $p$ disjoint conjugate disks from $X$ and gluing to each boundary component a copy of $Y\setminus D^2$. The result is a new space on which we can naturally define a $C_p$-action. This is called the equivariant connected sum of $X$ and $Y$ and is denoted $X\#_p Y$. An example of this operation is depicted in Figure \ref{cnctsum} in the case $p=3$. A precise definition of $X\#_p Y$ can be found in Section \ref{surgintro}. 

\begin{figure}
    \centering
    \includegraphics[scale=.5]{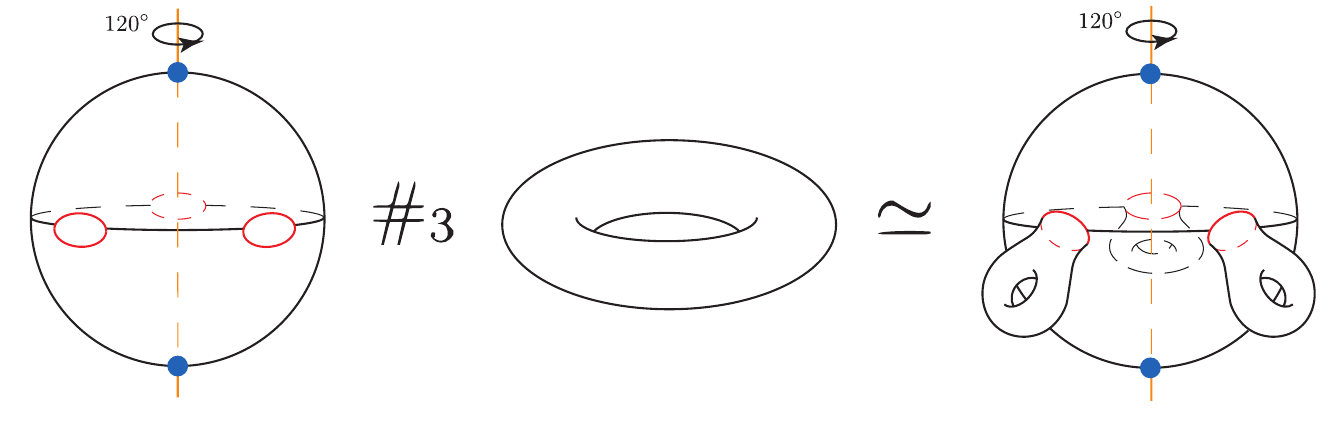}
    \caption{Equivariant connected sum surgery with $X=S^{2,1}$ and $Y$ a torus}
    \label{cnctsum}
\end{figure}

Let $R_p$  denote the space $S^{2,1}$ with $p$ disjoint conjugate disks removed. Let $X$ be any non-trivial $C_p$-surface. After removing $p$ disjoint conjugate disks from $X$, we can construct a new $C_p$-surface $X+[R_p]$ by gluing the $p$ boundary components of $X$ to those of $R_p$ via an equivariant map. Figure \ref{introsurgfig} depicts an example of this surgery operation in the case $p=3$. A precise definition of $X+[R_p]$ can be found in Section \ref{surgintro}. 

\begin{figure}
    \centering
    \includegraphics[scale=.5]{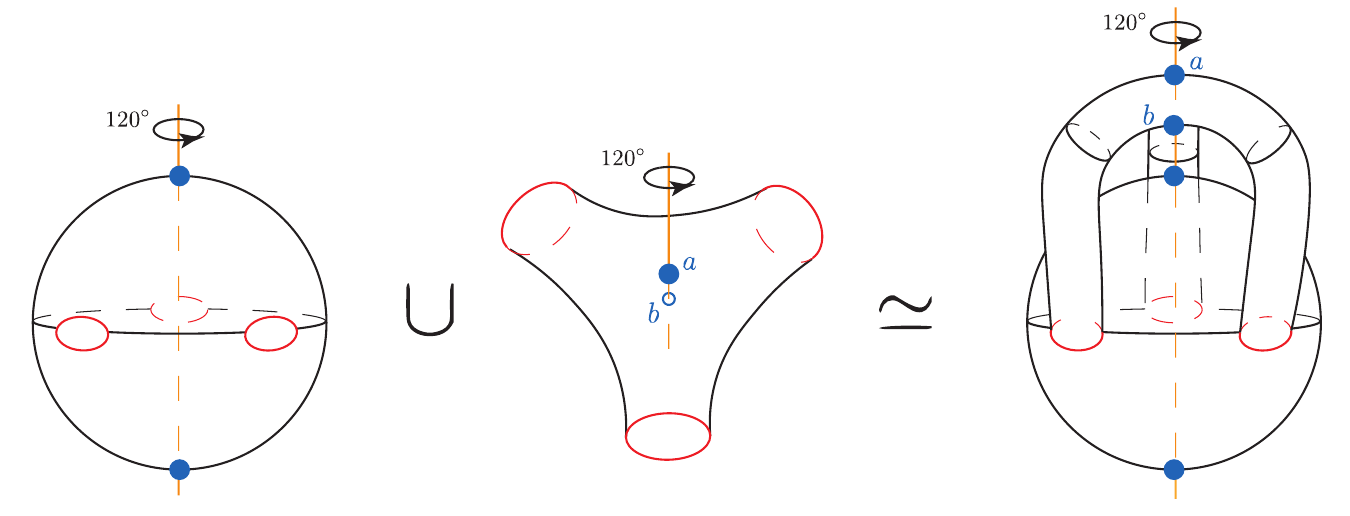}
    \caption{The space $S^{2,1}+[R_3]$}
    \label{introsurgfig}
\end{figure}

In this paper, we prove that up to isomorphism all $C_p$-surfaces can be constructed by starting with $M_1^{\text{free}}$, $S^{2,1}$, $N_2^{\text{free}}$, $N_1[1]$, or $\Hex_n$ (for some $n$) and performing a series of equivariant connected sum and ribbon surgeries. If $X$ is a surface with order $p$ homeomorphism $\sigma_X$ and $Y$ is a surface with order $p$ homeomorphism $\sigma_Y$, we say that $X$ and $Y$ are isomorphic if there exists a homeomorphism $f\colon X\rightarrow Y$ such that $f\circ\sigma_X=\sigma_Y\circ f$. 

Let $M_g$ denote the genus $g$, closed orientable surface.

\begin{theorem}
Let $X$ be a connected, closed, orientable surface with an action of $C_p$. Then $X$ can be constructed via one of the following surgery procedures, up to $\operatorname{Aut}(C_p)$ actions on each of the pieces.
\begin{enumerate}
	\item $M_1^{\text{free}}\#_pM_g$, $g\geq 0$
	\item $\left(S^{2,1}+k[R_p]\right)\#_pM_g$, $k,g\geq 0$
	\item $\left(\Hex_n+k[R_p]\right) \#_pM_g$, $k,g\geq 0$, $n\geq 1$
\end{enumerate}
\end{theorem}

Let $N_r$ denote the genus $r$, closed non-orientable surface.

\begin{theorem}
Let $X$ be a connected, closed, non-orientable surface with an action of $C_p$. Then $X$ can be constructed via one of the following surgery procedures, up to $\operatorname{Aut}(C_p)$ actions on each of the pieces.
\begin{enumerate}
	\item $N_2^{\text{free}}\#_pN_r$, $r\geq 0$
	\item $\left(S^{2,1}+k[R_p]\right)\#_pN_r$, $r\geq 1$
	\item $\left(N_1[1]+k[R_p]\right)\#_pN_r$, $k,r\geq 0$
\end{enumerate}
\end{theorem}

Unlike the corresponding result of Dugger, this classification does not provide a complete list of invariants distinguishing isomorphism classes. For example, we do not provide invariants with which to distinguish the $C_3$-surfaces $\Hex_2$ and $S^{2,1}+2[R_3]$. We instead prove that these spaces are non-isomorphic and that they represent the only closed and connected genus $6$ orientable $C_3$-surfaces with $6$ fixed points up to equivariant isomorphism. 


\subsection{Organization of the Paper}
Equivariant surgery procedures are outlined in Section \ref{surgintro}. Section \ref{thmsection} contains a statement of the main classification theorem for nontrivial $C_p$-surfaces. Some important equivariant surgery results are proved in Section \ref{invarianceresults}. A detailed proof of the main classification theorem from Section \ref{thmsection} is given in Sections \ref{freeproof} and \ref{nonfreeproof}.

\subsection{Acknowledgements}
The work in this paper was a portion of the author's thesis project at the University of Oregon. The author would first like to thank her doctoral advisor Dan Dugger for his invaluable guidance and support. The author would also like to thank Christy Hazel and Clover May for countless helpful conversations as well as Robert Lipshitz for many constructive comments. This research was partially supported by NSF grant DMS-2039316.  

\section{$C_p$-equivariant Surgeries of Surfaces}\label{surgintro}

Let $p$ be an odd prime. There are $(p-1)/2$ isomorphism classes of $C_p$-actions on $\R^2$ corresponding to rotation about the origin by a $p$th root of unity. Rotation of the plane by $\omega_i$ is isomorphic to rotation by $\omega_j$ only when $\omega_i=\overline{\omega_j}$. However if we consider such rotations up to an action of $\operatorname{Aut}(C_p)$, then we are left with only one isomorphism class of nontrivial actions on $\R^2$. In this section we lay the ground work for a classification of closed surfaces with a nontrivial action of $C_p$ up to an action of $\operatorname{Aut}(C_p)$. We do this by defining analogues of equivariant surgery methods from \cite{Dug19} in the odd prime case. 

For a $C_p$-surface $X$, let $F(X)$ denote the number of fixed points of $X$. It is useful to note that when the action is non-trivial, $F(X)$ must be finite. We also let $\beta(X)$ denote the \textbf{$\beta$-genus} of $X$, defined to be $\operatorname{dim}_{\Z/2} H^1_{\text{sing}}(X;\Z/2)$. 



\subsection{Equivariant Connected Sums}

\begin{definition}
Let $Y$ be a non-equivariant surface and $X$ a surface with a nontrivial order $p$ homeomorphism $\sigma\colon X\rightarrow X$. Define $\tilde{Y}:=Y\setminus D^2$, and let $D$ be a disk in $X$ so that $D$ is disjoint from each of its conjugates $\sigma^i D$. Similarly let $\tilde{X}$ denote $X$ with each of the $\sigma^i D$ removed. Choose an isomorphism $f\colon \partial \tilde{Y}\rightarrow \partial D$. We define an \textbf{equivariant connected sum} $X\#_p Y$, by 
\[\left[\tilde{X}\sqcup\coprod_{i=0}^{p-1}\left(\tilde{Y}\times \{i\}\right)\right]/\sim\]
where $(y,i)\sim \sigma^i(f(y))$ for $y\in \partial \tilde{Y}$ and $0\leq i \leq p-1$. We can see an example of this surgery in Figure \ref{equivcnctsum}.
\end{definition}

\begin{figure}
\begin{center}
\includegraphics[scale=.5]{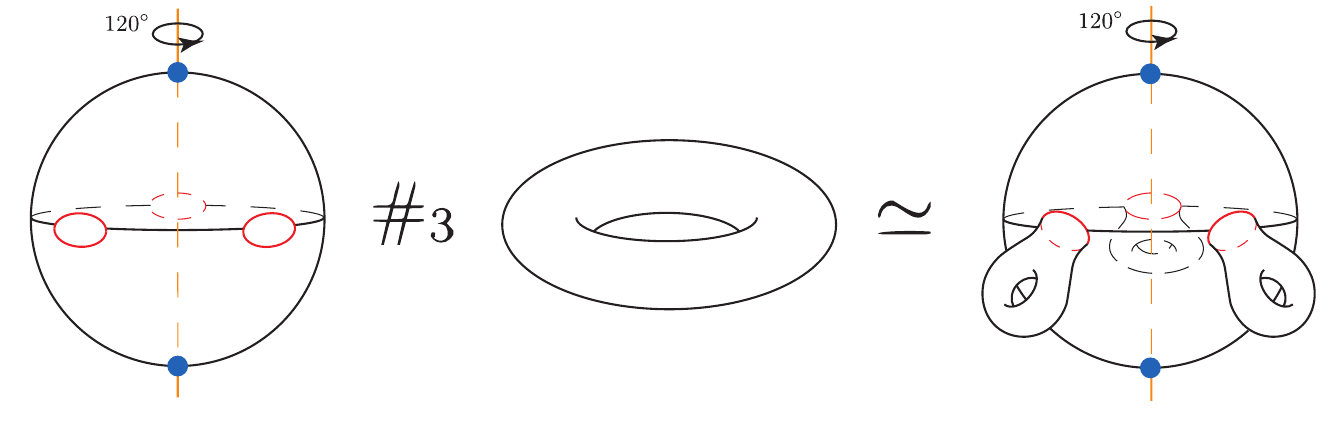}
\end{center}
\caption{\label{equivcnctsum} We can see above the result of the surgery $S^{2,1}\#_3 M_1$.}
\end{figure}

We will prove in Proposition \ref{welldefined} that the space $X\#_p Y$ is independent of the chosen disk $D$.

\begin{remark}
Any nontrivial $C_p$-surface has only a finite number of isolated fixed points since each fixed point must have a neighborhood isomorphic to $\R^2$ with a rotation action. 

For a $C_p$-space $X$ with $F$ fixed points and $\beta$-genus $\beta_1$ and a non-equivariant surface $Y$ with $\beta$-genus $\beta_2$, $X\#_p Y$ has $F$ fixed points and $\beta$-genus $\beta_1+p\beta_2$.
\end{remark}

\subsection{$C_p$-equivariant Ribbon Surgeries}
There are $(p-1)/2$ non-isomorphic $C_p$-actions on $S^2$ given by rotation by a primitive $p$th root of unity about the axis passing through its north and south poles. When the prime $p$ is understood, we let $S^{2,1}_{(i)}$ denote this sphere with rotation by $e^{2\pi i/p}$ where $1\leq i\leq p-1$. We additionally write $S^{2,1}$ when only considering such actions up to twisting by $\operatorname{Aut}(C_p)$. 

\begin{remark}
    The sphere $S^{2,1}_{(i)}$ is an example of a \textbf{representation sphere}. It can be defined as the one point compactification of a two-dimensional nontrivial $C_p$ representation. These objects are incredibly important in equivariant homotopy theory, so we choose our notation to be consistent with other papers in this field.
\end{remark}

\begin{definition}\label{ribbondef}
Let $D$ be a disk in $S^{2,1}_{(i)}$ that is disjoint from each of its conjugate disks. We define a \textbf{$C_p$-equivariant ribbon} as 
\[S^{2,1}_{(i)}\setminus \left(\coprod_{j=0}^{p-1} \sigma^j D\right),\]
and we denote this space $R_{p,(i)}$. We can see $R_{p,(1)}$ depicted in Figure \ref{R_p} in the cases $p=3$ and $p=5$. The action of $R_{p,(i)}$ can be described as rotation about the orange axis. There are two fixed points of this action, given by the points in blue where the axis of rotation intersects the surface.
\end{definition}

\begin{definition}
Let $X$ be a surface with a nontrivial order $p$ homeomorphism $\sigma\colon X\rightarrow X$. Choose a disk $D_1$ in $X$ that is disjoint from $\sigma^jD_1$ for each $j$. Then remove each of the $\sigma^jD_1$ to form the space $\tilde{X}$. As in Definition \ref{ribbondef}, let $D$ be the disk in $S^{2,1}_{(i)}$ which was removed (along with its conjugates) to form $R_{p,(i)}$. Choose an isomorphism $f\colon \partial D_1\rightarrow \partial D$ and extend this equivariantly to an isomorphism $\tilde{f}\colon \partial \tilde{X}\rightarrow\partial R_{p,(i)}$. We then define \textbf{$C_p$-ribbon surgery} on $X$ to be the space
\[\left(\tilde{X}\sqcup R_{p,(i)}\right)/\sim\]
where $x\sim \tilde{f}(x)$ for $x\in\partial\tilde{X}$. This is a new $C_p$-surface which we will denote $X+[R_{p,(i)}]$.
\end{definition}

\begin{remark}
There is an action of $\operatorname{Aut}(C_p)$ on $S^{2,1}_{(i)}$ (and thus $R_{p,(i)}$) given by $\sigma S^{2,1}_{(i)}=S^{2,1}_{(\sigma(i))}$ for $\sigma\in\operatorname{Aut}(C_p)$. 
Our goal is to classify all $C_p$-surfaces using equivariant surgery methods up to this action of $\operatorname{Aut}(C_p)$ on each of the surgery pieces. Going forward, we will use the notation $X+[R_p]$ to denote a $C_p$-surface obtained by performing some $C_p$-ribbon surgery on $X$. The notation $X+[R_p]$ therefore refers to several distinct isomorphism classes of $C_p$-surfaces which can be obtained from each other by the action of $\operatorname{Aut}(C_p)$ on each of the surgery pieces. We similarly let $S^{2,1}$ denote the $2$-sphere with a rotation action of $C_p$, noting that each of these can be obtained from the standard rotation of $e^{2\pi/p}$ by this action of $\operatorname{Aut}(C_p)$. 

In the $p=3$ case, this action of $\operatorname{Aut}(C_p)$ is trivial since $S^{2,1}_{(1)}\cong S^{2,1}_{(2)}$. Thus the notation $X+[R_3]$ (as well as $S^{2,1}$) is well-defined and denotes a single $C_3$-surface up to equivariant isomorphism.
\end{remark}

\begin{figure}
\begin{center}
\includegraphics[scale=.6]{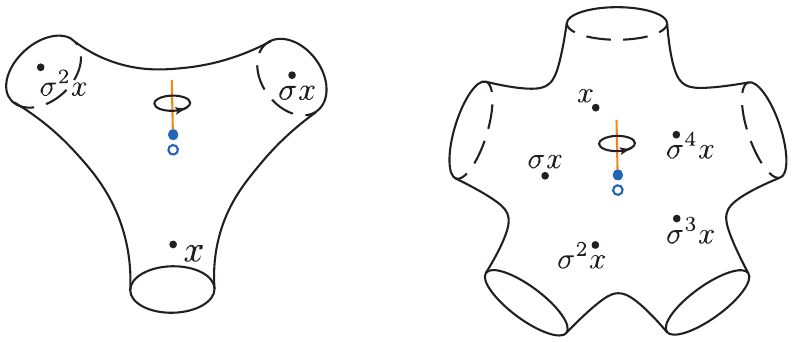}
\end{center}
\caption{The $C_p$-surface ${R_p}_{(i)}$ in the cases $p=3$ (left) and $p=5$ (right)\label{R_p}.}
\end{figure}

We will prove in Corollary \ref{independence of disk choice} that the space $X+[R_{p,(i)}]$ is independent of the chosen disk $D_1$.

For a $C_p$-surface $X$ with $F$ fixed points and $\beta$-genus $\beta$, the space $X+[R_{p,(i)}]$ has $F+2$ fixed points and $\beta$-genus $\beta+2(p-1)$. 

Let $X+k[R_{p,(i)}]$ denote the surface obtained by performing $C_p$-ribbon surgery $k$ times on $X$. We will see in Corollary \ref{independence of disk choice} that $+[R_{p,(i)}]$-surgery is independent of the choice of disk $D_1$. Because of this, $C_p$-ribbon surgery is associative and commutes with itself, making this notation well-defined.

\begin{definition}
We next define the $C_p$-surface $TR_{p,(i)}$ using a gluing diagram. Start with a $2p$-gon with a disk removed from its center. Then identify opposite edges of the $2p$-gon in the same direction to obtain the space $TR_{p,(i)}$. Figure \ref{TR3} shows this in the case $p=3$. The action on $TR_{p,(i)}$ is defined by rotation about its center by an angle corresponding to the $p$th root of unity $e^{2\pi i /p}$ ($1\leq i \leq (p-1)/2$). Note that $TR_{p,(i)}\cong TR_{p,(j)}$ only when $j=i$ or $j=p-i$. This surface is orientable with one boundary component.
\end{definition}

When $p=3$, $TR_{3,(1)}\cong TR_{3,(2)}$, so for simplicity of notation we will denote this space by $TR_3$.

\begin{figure}
\begin{center}
\includegraphics[scale=.6]{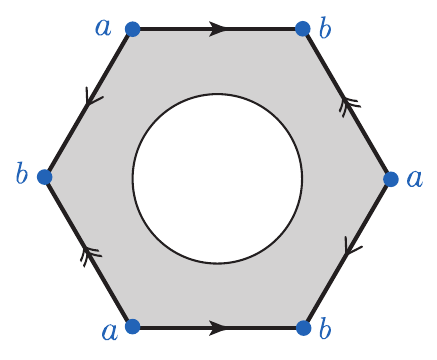}
\end{center}
\caption{\label{TR3} The $C_3$-surface $TR_3$.}
\end{figure}

\begin{lemma}
The surface $TR_{p,(i)}$ has two fixed points.
\end{lemma}

\begin{proof}
Consider the space $TR_{p,(i)}$ with its first $p$ edges labeled $e_1,\dots ,e_p$ as shown in Figure \ref{TRp}. Since opposite edges of the $2p$-gon are identified, all other edges are named accordingly. Let $v_1$ be the starting vertex of $e_1$, and let $v_2$ be the ending vertex of $e_1$. We first claim that all other vertices of the $2p$-gon representing $TR_{p,(i)}$ must be identified with either $v_1$ or $v_2$. Looking at the edge labeled $e_2$ towards the top of the polygon, we see that $e_2$ shares a starting vertex with $e_1$. Now looking at its opposite edge, it is also the case that $e_2$ shares an ending vertex with $e_1$. We can keep going to see that $e_3$ must share starting and ending vertices with $e_2$, and in fact all vertices $e_k$ must have starting vertex $v_1$ and ending vertex $v_2$. 

Finally observe that since the action of $C_p$ takes $e_1$ to $e_{k}$ for some $k$, the vertices $v_1$ and $v_2$ are fixed under the action. Thus, $TR_{p,(i)}$ has two fixed points.
\end{proof}

\begin{figure}
\begin{center}
\includegraphics[scale=.4]{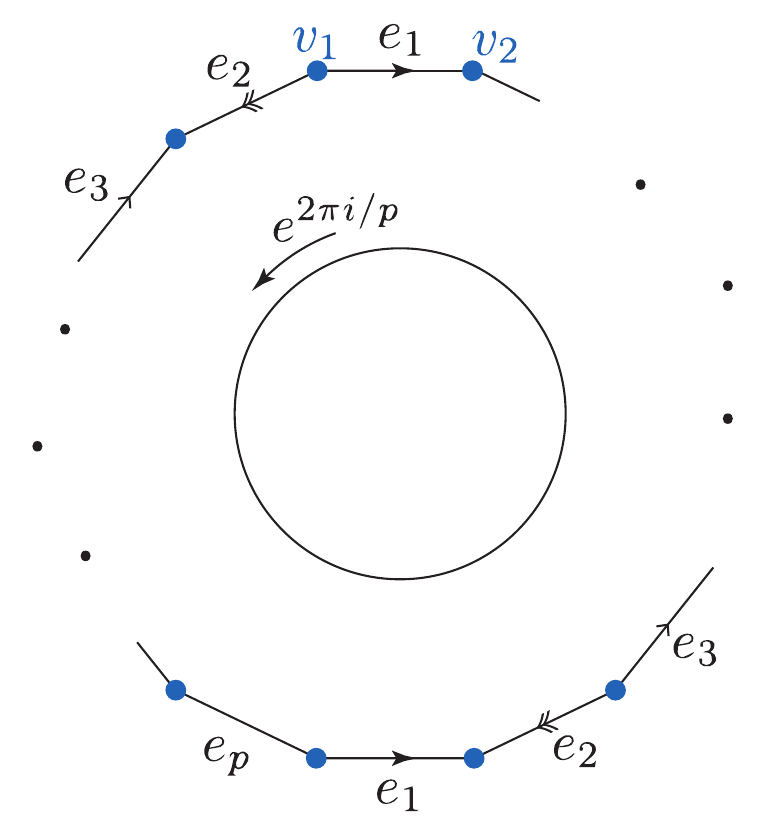}
\end{center}
\caption{\label{TRp} The surface $TR_{p,(i)}$.}
\end{figure}

\begin{definition}
Let $X$ be a non-trivial $C_p$-space with at least one isolated fixed point $x$. Choose a neighborhood $D_x$ of $x$ that is fixed by the action of $\sigma$. We then let $\tilde{X}$ denote $X\setminus D_x$. The action on the boundary of $\tilde{X}$ will be rotation by $e^{2\pi i/p}$ for some $i$. Fix an isomorphism $f\colon \partial\tilde{X}\rightarrow \partial TR_{p,(i)}$. The \textbf{$C_p$-twisted ribbon surgery} on $X$ is given by
\[\left(\tilde X\sqcup TR_{p,(i)}\right)/\sim\]
where $y\sim f(y)$ for $y\in\partial\tilde{X}$. We denote this new space by $X+_x[TR_p]$.
\end{definition}

For a $C_p$-surface $X$ with $F$ fixed points and $\beta$-genus $\beta(X)$, the space $X+_x[TR_p]$ has $F+1$ fixed points and $\beta$-genus $\beta(X)+2(p-1)$.

\begin{remark}
We will see in Corollary \ref{independence of disk choice} that $+[R_{p,(i)}]$-surgery does not depend on the initial disks chosen for the surgery, making the notation $X+[R_{p,(i)}]$ well defined. Unfortunately, the same is not true of twisted ribbon surgery. To specify our choice of initial fixed point $x$, we will use the notation $X+_x[TR_p]$. We will see in Example \ref{Bn} a space $X$ and choices of fixed points $x$ and $y$ where $X+_x[TR_p]\not\cong X+_y[TR_p]$.
\end{remark}

Let $X$ be a $C_p$-space with two distinct fixed points $x$ and $y$. By Proposition \ref{EB in X}, there exists a simple path $\alpha$ in $X$ from $x$ to $y$ that does not intersect its conjugate paths. Observe that the union of all conjugates of $\alpha$ is isomorphic to $EB_p$, where $EB_p$ denotes the unreduced suspension of $C_p$. In particular, given any $C_p$-space $X$ with at least two isolated fixed points, we can find a copy of $EB_p$ sitting inside $X$. We know from Lemma \ref{tube} that a neighborhood of this copy of $EB_p$ must be isomorphic to $R_{p,(i)}$ or $TR_{p,(i)}$. Given such a space, we can ``undo'' the corresponding ribbon surgery to construct a new space $X-[R_{p,(i)}]$ (respectively $X-[TR_p]$) which we define below. Figure \ref{ribbonswithEB} shows us how $R_{p,(i)}$ and $TR_{p,(i)}$ can be viewed as neighborhoods of $EB$. 

\begin{figure}
\begin{center}
\includegraphics[scale=.5]{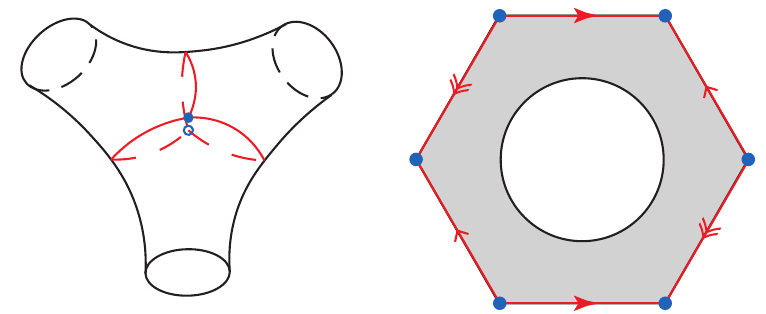}
\end{center}
\caption{\label{ribbonswithEB} The spaces $R_{p,(i)}$ (left) and $TR_{p,(i)}$ (right) in the case $p=3$, each containing $EB_p$ in red.}
\end{figure}

\begin{definition}
Let $X$ be a $C_p$-surface with isolated fixed points $a$ and $b$, and suppose the corresponding $EB_p$ containing $a$ and $b$ has a neighborhood homeomorphic to $R_{p,(i)}$. Then $\tilde{X}:=X\setminus R_{p,(i)}$ has $p$ boundary components, and there is an isomorphism $f\colon \partial \tilde{X}\rightarrow \partial\left(D^2\times C_p\right)$. Define $X-[R_{p,(i)}]$ to be 
\[\left(\tilde{X}\sqcup \left(D^2\times C_p\right)\right)/\sim\]
where $a\sim f(a)$ for $a\in\partial \tilde{X}$.
\end{definition}

As a result of this surgery, the space $X-[R_{p,(i)}]$ has $2$ fewer fixed points, and its $\beta$-genus is reduced by $2(p-1)$ from that of $X$. Moreover, if $X$ was a connected $C_p$-surface with at least $3$ fixed points, then $X-[R_{p,(i)}]$ is also connected. This does not have to be the case when $F=2$ however. For example, there exists $EB_p\subseteq S^{2,1}$ such that $\left(S^{2,1}\#_p M_1\right) -[R_p] \cong M_1\times C_p$. 

Let $a,b\in X$ be fixed points such that $a$ and $b$ live in some copy of $TR_{p,(i)}$ inside of $X$. We can similarly define $X-_{a,b}[TR_p]$ to be the result of surgery which removes this copy of $TR_{p,(i)}$ from $X$ and glues in $D^{2,1}$ along the boundary. As one would expect, the space $X-_{a,b}[TR_p]$ has one fewer fixed point and $\beta$-genus $p-1$ smaller than $X$. 

\begin{remark}
Although by Corollary \ref{independence of disk choice} we know $+[R_{p,(i)}]$ is independent of the disks chosen, $-[R_{p,(i)}]$ surgery does depend on a choice of $EB_p$. Two different choices of $R_{p,(i)}$ in a space can result in different spaces once $-[R_{p,(i)}]$ is performed. As a result, the notation $X-[R_{p,(i)}]$ is not well defined. Going forward, we will use the notation $X-[R_p]$ when the choice of $R_{p,(i)}$ is understood. Figure \ref{-Rpchoices} shows this using the example $S^{2,1}\#_3 M_1$. For the choice of $EB$ on the left, $-[R_3]$ surgery results in the space $M_1\times C_3$. For the choice on the right, $-[R_3]$ surgery results in the space $M_1^{\text{free}}$
\end{remark}

\begin{figure}
\begin{center}
\includegraphics[scale=.3]{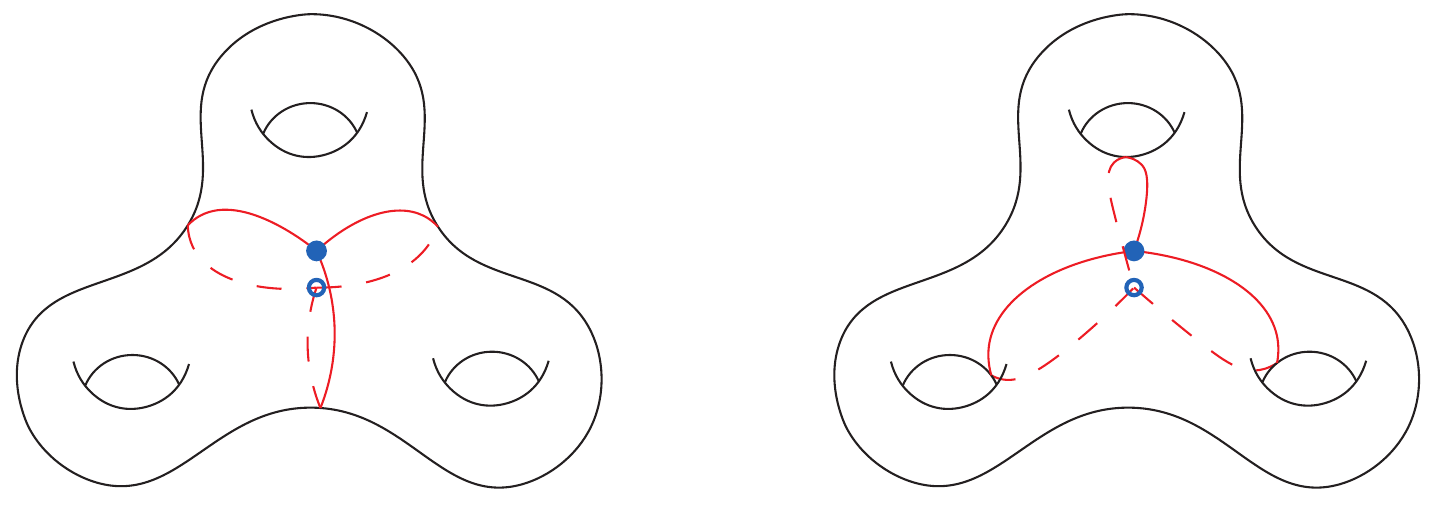}
\end{center}
\caption{\label{-Rpchoices} Two choices of $EB$ in $S^{2,1}\#_3 M_1$.}
\end{figure}

\begin{proposition}
Let $X$ be a non-trivial $C_p$-surface and $Y$ a non-equivariant surface. Then $\left(X+[R_{p,(i)}]\right)\#_pY\cong \left(X\#_pY\right)+[R_{p,(i)}]$. If $X$ has a fixed point $x$, it is also true that $\left(X+_x[TR_p]\right)\#_pY\cong \left(X\#_pY\right)+_x[TR_p]$.

Additionally, if $X$ is a space for which $-[R_p]$ or $-[TR_p]$-surgeries are defined, then $\left(X-[R_p]\right)\#_pY\cong \left(X\#_pY\right)-[R_p]$ (respectively $\left(X-[TR_p]\right)\#_pY\cong \left(X\#_pY\right)-[TR_p]$).
\end{proposition}

In other words,the equivariant connected sum surgery operation commutes with $\pm [R_{p,(i)}]$ and $\pm [TR_p]$ on all $C_p$-surfaces $X$ for which these surgeries are defined. 

In the case of $-[R_p]$ or $\pm [TR_p]$ surgeries, this is clear because these surgery operations take place in the neighborhood of fixed points while we can choose to perform any equivariant connected sum operation away from these fixed points. The proof that equivariant connected sum surgery commutes with $+[R_{p,(i)}]$-surgery is similar to the argument presented in the proof of Corollary \ref{independence of disk choice} and is left to the reader.

\subsection{M\"{o}bius Band Surgeries}

\begin{definition}
Represent the M\"{o}bius band as the usual quotient of the unit square where $(0,y)\sim (1,1-y)$. We define $(p-1)/2$ actions of $C_p$ on the m\"{o}bius band as follows. For a generator $\sigma$ of $C_p$, let $\sigma(x,y)=\left(x+\frac{i}{p},1-y\right)$ for $1\leq i \leq (p-1)/2$. Denote this space $MB_{p,(i)}$.  Figure \ref{mobiusband} gives a visual representation of this action in the case $p=3$.
\end{definition}

Note that the action on the boundary of $MB_{p,(i)}$ is the rotation action of $S^1$ by $e^{-2\pi i/p}$. 

When $p=3$, $MB_{3,(1)}\cong MB_{3,(2)}$, so for simplicity of notation we will denote this space by $MB_3$.

\begin{figure}
\begin{center}
\includegraphics[scale=.5]{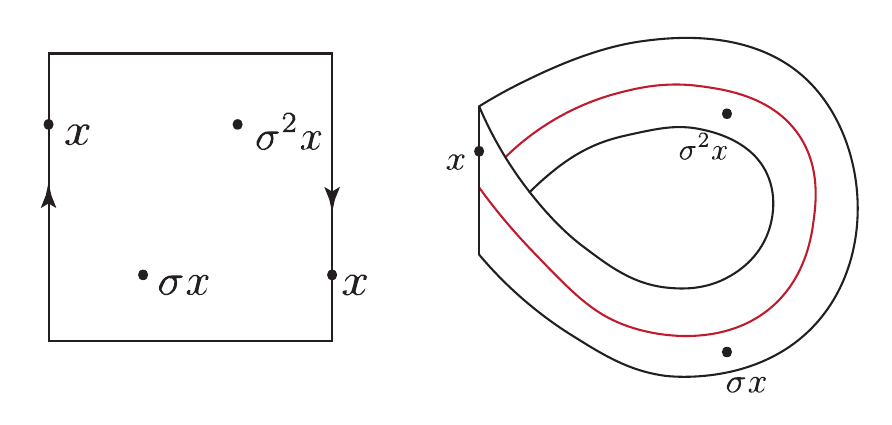}
\end{center}
\caption{\label{mobiusband} The $C_3$-space $MB_3$, whose underlying space is the m\"{o}bius band.}
\end{figure}

\begin{definition}
Let $X$ be a non-trivial $C_p$-surface with fixed point $x$. Choose a neighborhood $D_x$ of $x$ which is fixed under the action of $\sigma\in C_p$. The $C_p$-space $\tilde{X}:=X\setminus D_x$ has a distinguished boundary component isomorphic to $S^1$ with rotation by some angle $e^{2\pi i/p}$. Fix an equivariant isomorphism $f\colon \partial \tilde{X}\rightarrow \partial MB_{p,(i)}$. We can then define a new $C_p$-space
\[\left(\tilde{X}\sqcup MB_{p,(i)}\right)/\sim\]
where $x\sim f(x)$ for $x\in \partial \tilde{X}$. Denote this new space by $X+_x[FMB_p]$. This process is called \textbf{fixed point to m\"{o}bius band surgery}.
\end{definition}

\begin{remark}
Given a $C_p$ space $X$ with $F$ fixed points and $\beta$-genus $\beta$, the space $X+[FMB_p]$ has $F-1$ fixed points and genus $\beta+1$.
\end{remark}

\begin{definition}
We can similarly define \textbf{m\"{o}bius band to fixed point surgery} on a $C_p$ space $X$ with $MB_{p,(i)}\subseteq X$. This procedure is the reverse process of $+_x[FMB_p]$ surgery in the sense that it removes $MB_p$ from $X$ and glues in a copy of $D^{2,1}$ along the boundary. The resulting space is denoted $X+[MB_pF]$. This notation will only be used when the choice of m\"{o}bius band is understood.
\end{definition}

\section{Examples in the $p=3$ Case}

In this section we will highlight some of the surfaces we can now build using equivariant surgery. Although each of the following examples have analogues for higher $p$, we will focus mainly on the $p=3$ case.

\begin{example}[Free Torus]\label{freetorus}
There is a free $C_3$-action on the torus $M_1$ given by rotation of $120^\circ$ about its center. Denote this $C_3$-space by $M_1^{\text{free}}$. From this, we can perform an equivariant connected sum operation with the $g$-holed torus $M_g$ to construct the space $M_{3g+1}^{\text{free}}:=M_1^{\text{free}}\#_3 M_g$. The result is a free $C_3$-action on the $(3g+1)$-holed torus (ie. the orientable surface with beta genus $\beta=6g+2$). We will see in the next section that up to equivariant isomorphism there is only one free action of $C_3$ on $M_{3g+1}$. The space $M_{3g+1}^{\text{free}}$ can be seen in Figure \ref{free7torus} in the case $g=2$.
\end{example}

\begin{figure}
\begin{center}
\includegraphics[scale=.4]{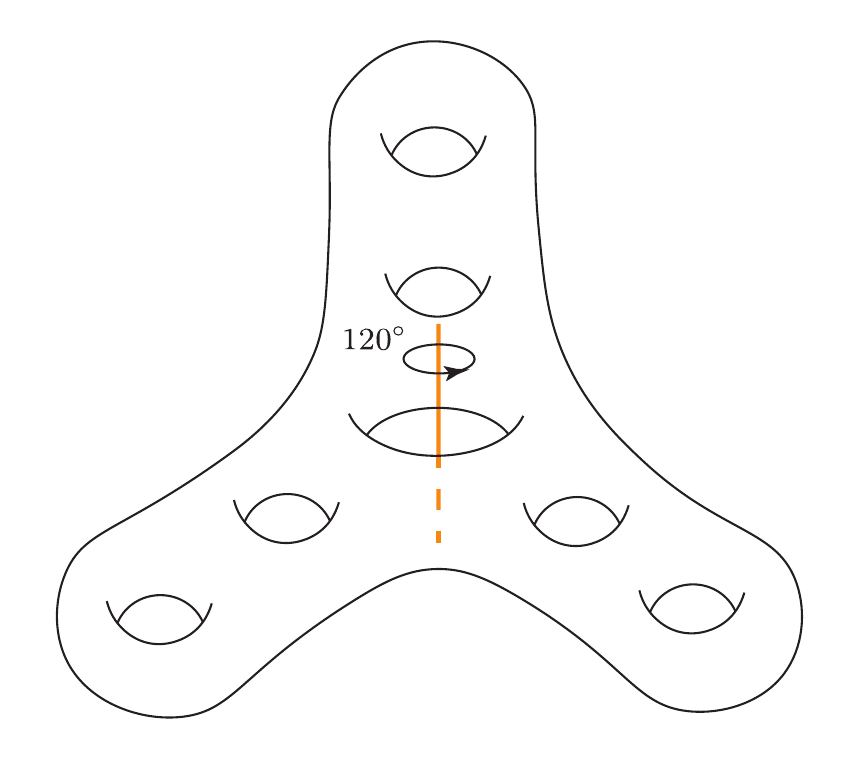}
\end{center}
\caption{\label{free7torus} The space $M_7^{\text{free}}$.}
\end{figure}

\begin{example}[$\Sph_g{[F]}$]
The representation sphere $S^{2,1}$ is defined as the $2$-sphere with a rotation action of $120^\circ$ about the axis passing through the north and south poles of the sphere. Since ribbon surgery and connected sum surgery commute with each other, we can consider the space $\Sph_{2k+3g}[2k+2]:=\left(S^{2,1}+k[R_3]\right)\#_3M_g$ which is constructed by performing ribbon surgery $k$ times on $S^{2,1}$ and then performing connected sum surgery with the orientable surface $M_g$. 

The space $\Sph_{2k+3g}[2k+2]$ has $2k+2$ fixed points and is non-equivariantly isomorphic to $M_{2k+3g}$.
\end{example}

\begin{example}[Non-free Torus]
Let $\Hex_1$ (Figure \ref{S(2,1)+[TR3]}) denote the space $S^{2,1}+_S[TR_3]$ where $S$ denotes the south pole of $S^{2,1}$. Then $\Hex_1$ has $\beta$-genus $\beta=2$ and $3$ fixed points. We can additionally observe that the space $S^{2,1}+_N[TR_3]$ (where $N$ is the north pole this time) is isomorphic to $\Hex_1$. 

For $p>3$, we can define similar spaces denoted ${\Hex_1^p}_{(i)}$ (where the action is given by the usual rotation by a $p$th root of unity about the center). The underlying space is a genus $(p-1)/2$ orientable surface represented by a $2p$-gon with appropriate identifications. The notation $\Hex$ was chosen to remark on the fact that this surface is most easily seen through this polygonal representation of $M_{(p-1)/2}$.
\end{example}

\begin{figure}
\begin{center}
\includegraphics[scale=.5]{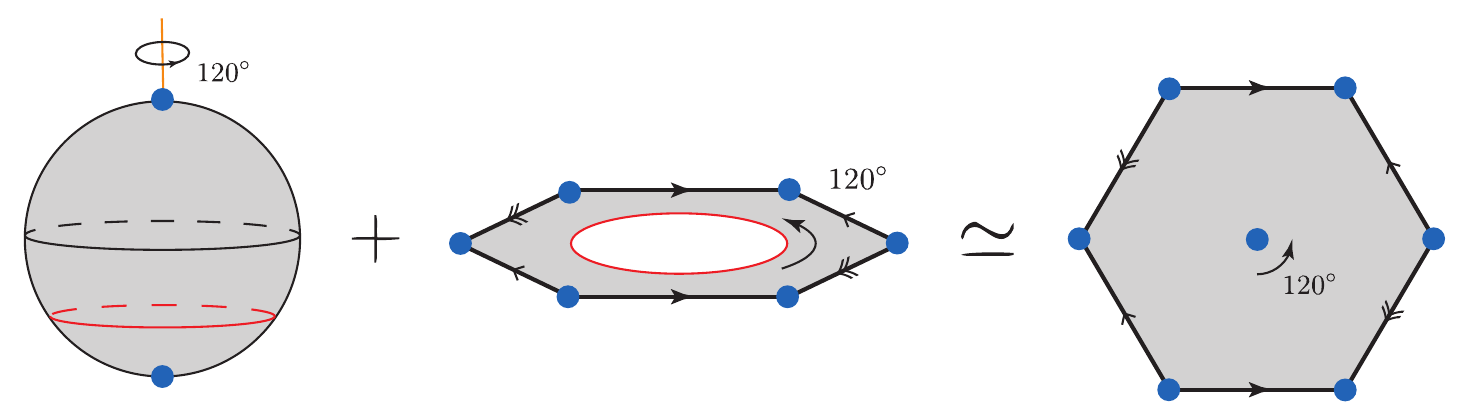}
\end{center}
\caption{\label{S(2,1)+[TR3]} The space $\Hex_1=S^{2,1}+_S [TR_3]$.}
\end{figure}

\begin{example}[$\Hex_n$]\label{Bn}
Consider the surface $\Hex_1+[R_3]$ with $\beta$-genus $\beta=6$ and $F=5$ fixed points. Label the fixed points as shown in Figure \ref{Bnsetup}. As a result of Lemma \ref{M1[3]+[TAT]}, we know that $+_{c_i}[TR_3]$-surgery results in a space isomorphic to $S^{2,1}+2[R_3]$. One naturally asks the question: Does twisted ribbon surgery yield the same space when centered around the fixed points $a$ or $b$? As it turns out, we get the same result after performing $+_a[TR_3]$-surgery, but ribbon surgery centered on the point $b$ yields a different $C_3$-surface. This new surface (which we will call $\Hex_2$) is depicted in Figure \ref{Bnsetup}. Proposition \ref{hex2} contains the proof of the fact that $\Hex_2$ and $S^{2,1}+2[R_3]$ are non-isomorphic surfaces. 

\begin{figure}
\begin{center}
\includegraphics[scale=.5]{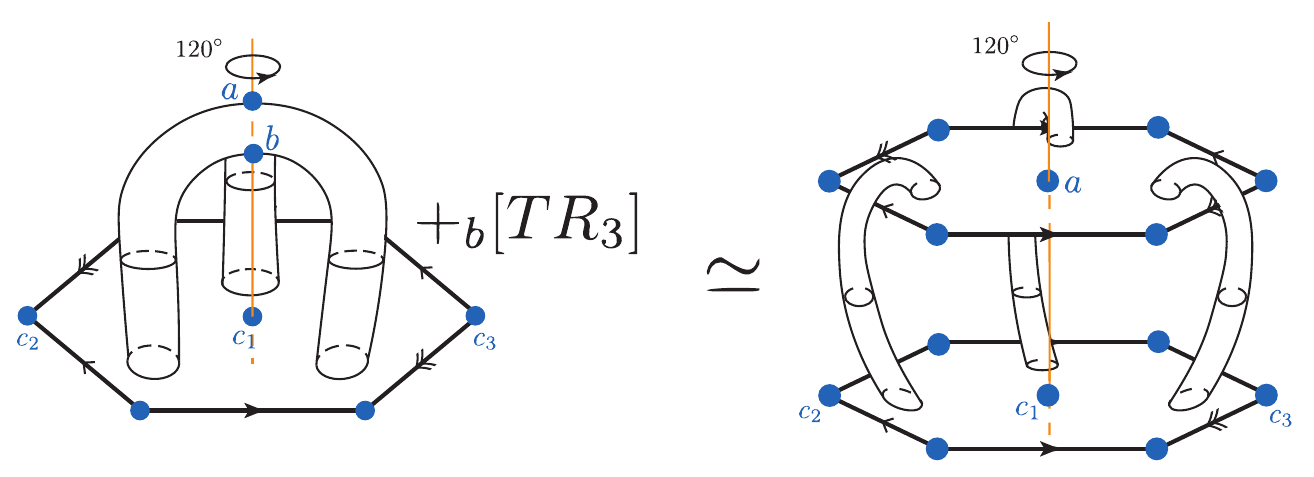}
\end{center}
\caption{\label{Bnsetup} Twisted ribbon surgery centered on $b$ yields the space $\Hex_2$.}
\end{figure}

Now that we have a new $C_3$-surface $\Hex_2$, we can construct surfaces of the form $\Hex_2+k[R_3]\#_3 M_g$ for some $k,g\geq 0$. This brings us back to our previous question. What if we performed twisted ribbon surgery on $\Hex_2+k[R_3]\#_3M_g$? Does the result depend on the chosen fixed point? Ultimately, the answer depends on $k$. When $k=0$, twisted ribbon surgery is independent of the chosen fixed point. This is not true when $k>0$ however. In this case there are two isomorphism classes of spaces which can be obtained by performing twisted ribbon surgery on $\Hex_2+k[R_3]\#_3M_g$. We prove these facts in Section \ref{nonfreeproof}. 

For now, let us examine this through the $k=1$, $g=0$ case. The space $\Hex_2+[R_3]$ is shown on the left of Figure \ref{B2toB3}. Performing twisted ribbon surgery centered on any point other than $b$ results in the space $\Hex_1+3[R_3]$. However $+_b[TR_3]$-surgery produces a different space which we will call $\Hex_3$ (the space on the right of Figure \ref{B2toB3}). 

In general, we can inductively define a space $\Hex_n$ by starting with the space $\Hex_{n-1}+[R_3]$ and performing twisted ribbon surgery centered on a specific fixed point. Just as $\Hex_3$ is represented in Figure \ref{B2toB3} as a tower of three hexagons, the space $\Hex_n$ for $n\geq 1$ can be thought of as a tower of $n$ hexagons connected in a similar way. 

An analogous collection of $C_p$-spaces (for $p>3$) can be defined and will be denoted $\Hex_n^p$ when the prime $p$ is not understood. The $C_p$-space $\Hex_n^p$ has $3n$ fixed points and $\beta$-genus $\beta=(3n-2)(p-1)$. We can again most easily visualize this space as a tower of $n$ polygons. 

\begin{example}[Free Klein Bottle]\label{freekleinbottle}
The representation sphere $S^{2,1}$ has two fixed points, so we can consider the space $S^{2,1}+2[FMB_3]:=\left(S^{2,1}+_N[FMB_3]\right)+_S [FMB_3]$ where we perform $+[FMB_3]$ surgery on both the north and south poles. The resulting space must be free of $\beta$-genus $\beta=2$. We denote this free Klein Bottle by $N_2^{\text{free}}$. 

Other free non-orientable surfaces can be constructed by performing equivariant connected sum surgery on $N_2^{\text{free}}$. We will see in the next section that up to isomorphism there is only one free action on $N_{2+3r}$ for each $r\geq 0$, namely
\[N_{2+3r}^{\text{free}}:=N_2^{\text{free}}\#_3N_r.\]
\end{example}

\begin{figure}
\begin{center}
\includegraphics[scale=.5]{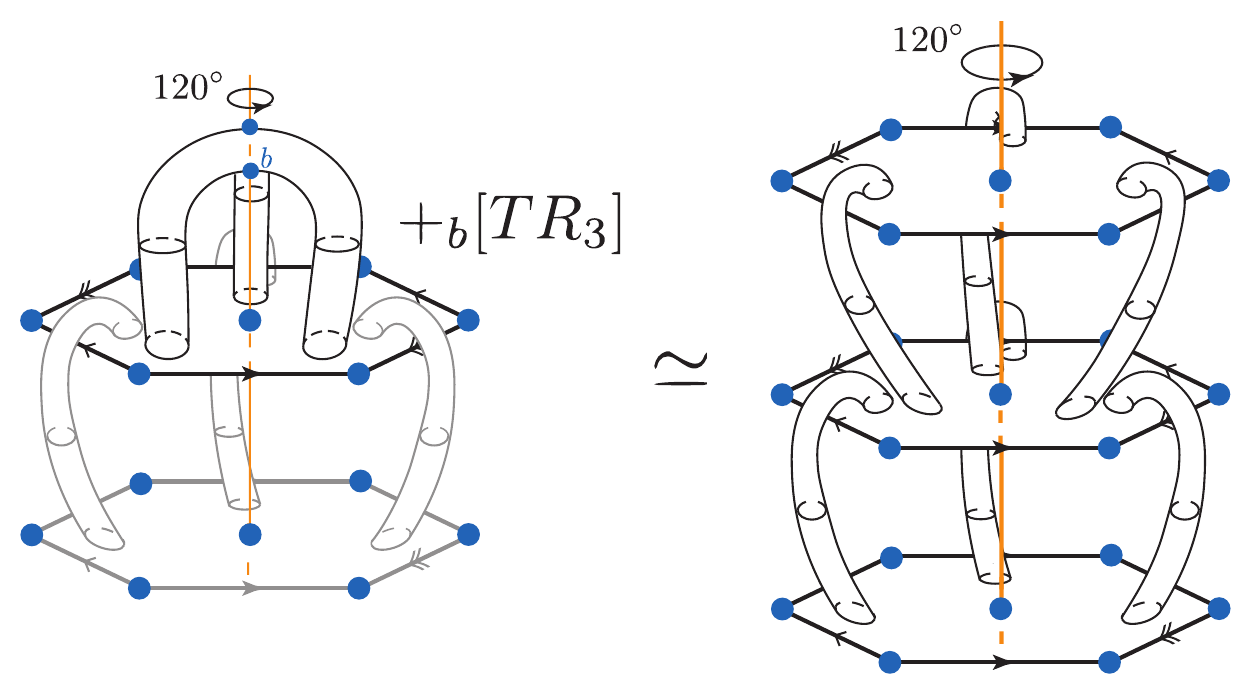}
\end{center}
\caption{\label{B2toB3} Twisted ribbon surgery centered on $b$ yields the space $\Hex_3$.}
\end{figure}
\end{example}

\section{Classifying $C_p$ Actions}\label{thmsection}

In this section we state the main classification theorem for nontrivial, closed surfaces with an action of $C_p$ for any odd prime $p$. All surfaces are defined up to an action of $\text{Aut}(C_p)$ on each of the surgery pieces. 

The proof of the classification of free $C_p$-surfaces can be found in Section \ref{freeproof}, while the proof of the non-free case is in Section \ref{nonfreeproof}.

\begin{lemma}\label{euler characteristic}
Let $X$ be a surface with beta genus $\beta$. If $\sigma\colon X\rightarrow X$ is a $C_p$-action with $F$ fixed points, then $F\equiv 2-\beta\pmod{p}$.
\end{lemma}

\begin{proof}
The space $X\setminus X^{C_p}$ is a free $C_p$-space with Euler characteristic $2-\beta-F$. Since the action is free, $X\setminus X^{C_p}\rightarrow(X\setminus X^{C_p})/C_p$ is a $p$-fold covering space. In particular, the Euler characteristic of $X\setminus X^{C_p}$ must be a multiple of $p$.
\end{proof}

\begin{theorem}\label{orientableclassification}
Let $X$ be a connected, closed, orientable surface with an action of $C_p$. Then $X$ can be constructed via one of the following surgery procedures, up to $\operatorname{Aut}(C_p)$ actions on each of the pieces:
\begin{enumerate}
	\item $M_{1+pg}^{\text{free}}:= M_1^{\text{free}}\#_pM_g$, $g\geq 0$
	\item $\Sph_{(p-1)k+pg}[2k+2]:=\left(S^{2,1}+k[R_p]\right)\#_pM_g$, $k,g\geq 0$
	\item $\Hex_{n,(3n-2)(p-1)/2+(p-1)k+pg}[3n+2k]:=\left(\Hex_n+k[R_p]\right) \#_pM_g$, $k,g\geq 0$, $n\geq 1$
\end{enumerate}
\end{theorem}

\begin{theorem}\label{nonorientableclassification}
Let $X$ be a connected, closed, non-orientable surface with an action of $C_p$. Then $X$ can be constructed via one of the following surgery procedures, up to $\operatorname{Aut}(C_p)$ actions on each of the pieces:
\begin{enumerate}
	\item $N_{2+pr}^{\text{free}}\cong N_2^{\text{free}}\#_pN_r$, $r\geq 0$
	\item $N_{2(p-1)k+pr}[2k+2]\cong\left(S^{2,1}+k[R_p]\right)\#_pN_r$, $r\geq 1$
	\item $N_{1+2(p-1)k+pr}[1+2k]\cong\left(N_1[1]+k[R_p]\right)\#_pN_r$, $k,r\geq 0$
\end{enumerate}
\end{theorem}

\begin{remark}\label{badnews}
It is important to note that for orientable surfaces, $\beta$ and $F$ do not provide enough information to distinguish between these families of isomorphism classes. For example, when $p=3$, $\Hex_{2,4}[6]$ and $\Sph_{4}[6]$ are non-isomorphic orientable surfaces with $\beta=4$ and $F=6$. See Proposition \ref{hex2} for a proof of this fact. 

In the case of non-orientable surfaces, $F$ and $\beta$ do distinguish between these families. In other words, given a non-orientable surface $X$ with specific values for $F$ and $\beta$, one can explicitly determine how $X$ was constructed via equivariant surgeries.
\end{remark}

Some examples of spaces in each of these families are shown in the case $p=3$ in Figures \ref{freespaces}, \ref{nonfreeorientable}, \ref{nonfreenonorientable}. 

\begin{figure}
\begin{center}
\includegraphics[scale=.3]{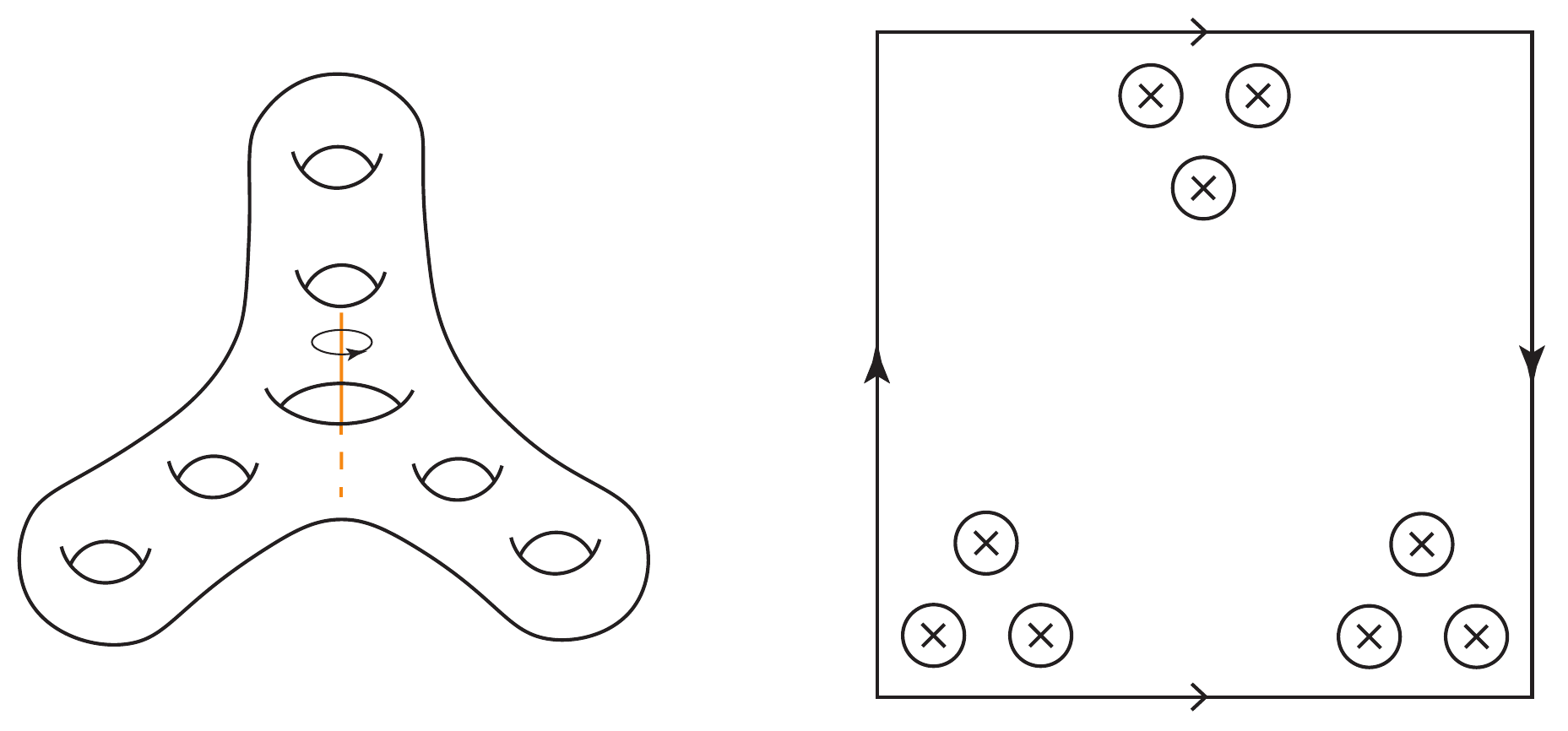}
\end{center}
\caption{\label{freespaces} The $C_3$-spaces $M_7^{\text{free}}$ (left) and $N_{11}^{\text{free}}$ (right).}
\end{figure}

\begin{figure}
\begin{center}
\includegraphics[scale=.4]{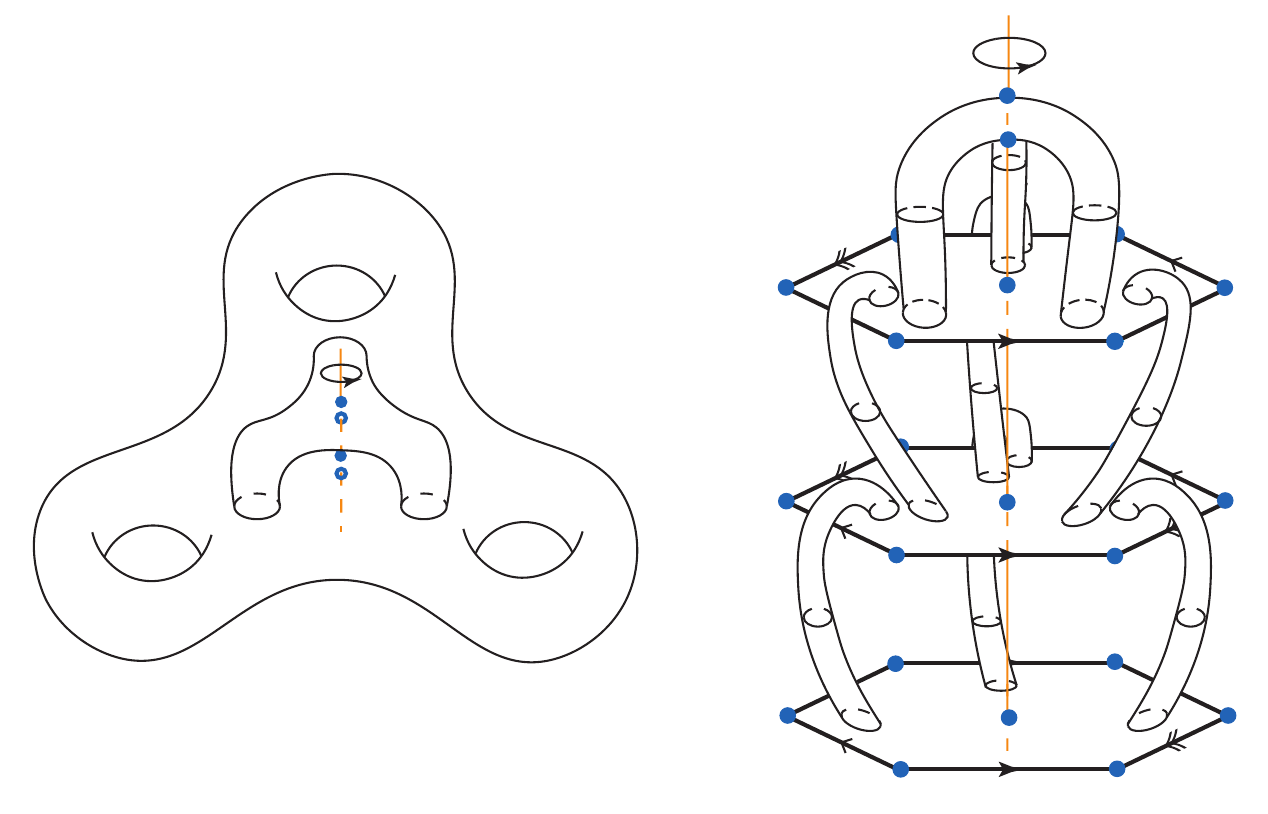}
\end{center}
\caption{\label{nonfreeorientable} The non-free $C_3$-spaces $S^{2,1}+[R_3]\#_3M_1$ (left) and $\Hex_3+[R_3]$ (right). Fixed points are shown in blue.}
\end{figure}

\begin{figure}
\begin{center}
\includegraphics[scale=.5]{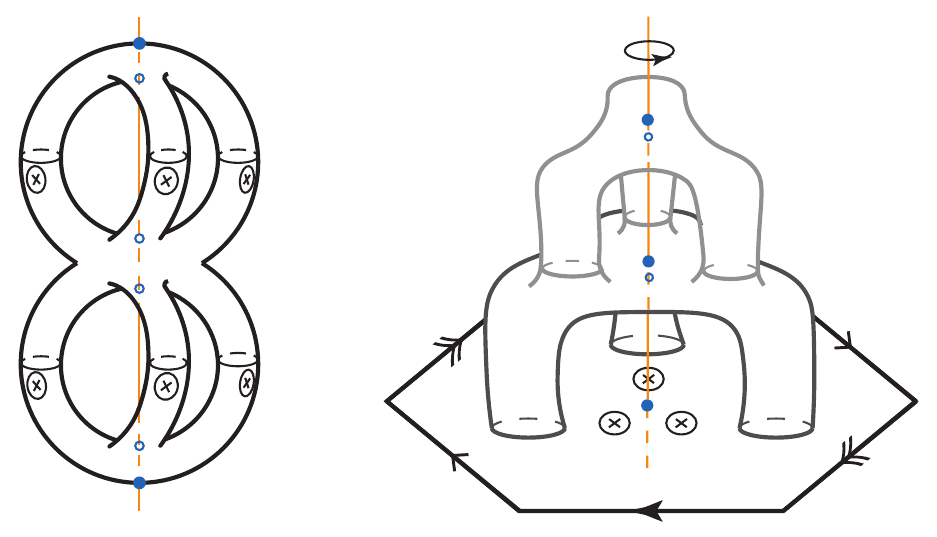}
\end{center}
\caption{\label{nonfreenonorientable} The $C_3$-spaces $ S^{2,1}+2[R_3]\#_3N_2$ (left) and $N_1[1]+2[R_3]\#_3N_1$ (right).}
\end{figure}

\section{Surgery Invariance Results}\label{invarianceresults}

Let $p$ be an odd prime. This section contains proofs for some of the basic surgery invariance results outlined in Section \ref{surgintro}.

\begin{proposition}\label{EB in X}
Let $X$ be a closed, connected 2-manifold with a map $\sigma\colon X\rightarrow X$ such that $\sigma^p=1$. Let $a,b\in X\setminus X^{C_p}$ such that $a\neq \sigma^k b$ for any $k$. Then there exists a simple path $\alpha$ in $X$ from $a$ to $\sigma^k b$ for some $k$ such that $\alpha$ does not intersect any of its conjugate paths. In other words, $\alpha(s)\neq \sigma^k\alpha(t)$ for all $k$, $s$, and $t$ ($k\neq 0$ if $s=t$).
\end{proposition}

\begin{proof}
    Choose an embedded path $\alpha$ in $\left(X\setminus X^{C_p}\right)/C_p$ from the image of $a$ to the image of $b$. The preimage of $\alpha$ in $X\setminus X^{C_p}$ consists of $p$ disjoint conjugate paths from $\sigma^ia$ to $\sigma^jb$. In particular, there is a component of this preimage which is a path from $a$ to $\sigma^kb$ for some $k$ with the desired property.
\end{proof}

\begin{corollary}\label{independence of disk choice}
Let $X$ be a path-connected, closed 2-manifold with a $C_p$ action. Let $Y_1$ be obtained from $X$ by removing disjoint conjugate disks embedded in $X\setminus X^{C_p}$ and sewing in a $C_p$-ribbon. Let $Y_2$ be similarly obtained from $X$, but using a different set of conjugate embedded disks. Then $Y_1\cong Y_2$.
\end{corollary}

\begin{proof}
Let $D_i, \sigma D_i,\dots ,\sigma^{p-1}D_i$ be the names of the disjoint disks removed to make $Y_i$ from $X$. Let $a_i$ denote the center of $D_i$. Then by Proposition \ref{EB in X} there is a path $\alpha$ from $a_1$ to $\sigma^k a_2$ for some $k$ that does not intersect its conjugate paths. From here, we can obtain an equivariant homeomorphism $X\rightarrow Y$ by following a nearly identical procedure to the proof of Corollary A.3 in \cite{Dug19}.
\end{proof}

\begin{proposition}\label{welldefined}
Let $X$ be a path-connected, closed 2-manifold with a $C_p$ action, and let $M$ be a non-equivariant connected surface. The equivariant isomorphism type of $X\#_p M$ is independent of the choice of disks used in the construction.
\end{proposition}

The proof of this proposition is nearly identical to that of Corollary \ref{independence of disk choice}.

\begin{proposition}\label{-[Rp] invariance}
Let $X$ and $Y$ be equivariant 2-manifolds that both contain a $C_p$-ribbon. If $X-[R_p]\cong Y-[R_p]$, then $X\cong Y$.
\end{proposition}

An analogous statement and proof of this fact for the $p=2$ case can be found in Proposition 3.11 of \cite{Dug19}.

\section{Free Classification Proof}\label{freeproof}

In order to prove Theorems \ref{orientableclassification} and \ref{nonorientableclassification}, we will induct on the number of fixed points of a given $C_p$-surface. In this section we prove the base case for this argument, that every closed surface with a free $C_p$ action is either isomorphic to $M_{1+pg}^{\text{free}}$ for some $g$, to $N_{2+pr}^{\text{free}}$ for some $r\geq 1$, or to $N_2$ with one of its $(p-1)/2$ free $C_p$-actions. 

Let $X$ be a path-connected non-equivariant space. Let $\mathcal{S}_p(X)$ denote the set of isomorphism classes of free $C_p$-spaces $Y$ that are path-connected and have the property that $Y/C_p \cong X$.

\begin{proposition}\label{freestuff}
There is a bijection between $\mathcal{S}_p(X)$ and the set of nonzero orbits in $H^1_{\text{sing}}(X;\Z/p)/\operatorname{Aut}(X)$.
\end{proposition}

An analogous proof of this fact for the $p=2$ case is provided in \cite{Dug19}, but we will summarize the main idea here. Given an element $Y$ of $\mathcal{S}(X)$, we get a principal $\mathbb{Z}/p$ bundle $Y\rightarrow X$ by choosing an isomorphism $Y/C_p\rightarrow X$. This then corresponds to an element of $H^1(X;\Z/p)$ via its characteristic class. To make this association well-defined, we must quotient out by the automorphisms of $X$. 

With this proposition, our goal is now to understand the action of $\operatorname{Aut}(X)$ on $H_{\text{sing}}^1(X;\Z/p)$. This is given by a group homomorphism
\[\operatorname{Aut}(X)\rightarrow\operatorname{Aut}(H_{\text{sing}}^1(X;\Z/p)).\]
Recall that the \textbf{full mapping class group} $\mathcal{M}(X)$ of a space $X$ is defined to be $\mathcal{M}(X)=\operatorname{Aut}(X)/\mathcal{I}(X)$ where $\mathcal{I}(X)$ is the subgroup of automorphisms that are isotopic to the identity. Since $\mathcal{I}(X)$ acts trivially on $H_{\text{sing}}^1(X;\Z/p)$, our action $\operatorname{Aut}(X)\rightarrow\operatorname{Aut}(H_{\text{sing}}^1(X;\Z/p))$ descends to a map
\[\mathcal{M}(X)\rightarrow \operatorname{Aut}(H_{\text{sing}}^1(X;\Z/p)).\]

\subsection{Non-orientable Case}

Since homology and cohomology are dual in $\Z/p$ coefficients, it is sufficient to consider the action of $\mathcal{M}(X)$ on $H_1(X;\Z/p)$. The space $X$ can be represented as a sphere with $r$ crosscaps $\alpha_1,\dots ,\alpha_r$ as in Figure \ref{Nr}, which we can choose as generators for $H_1(X;\Z/p)$. We begin by discussing generators of $\mathcal{M}(X)$ and how they act on the $\alpha_i$.  

\begin{figure}
\begin{center}
\includegraphics[scale=.35]{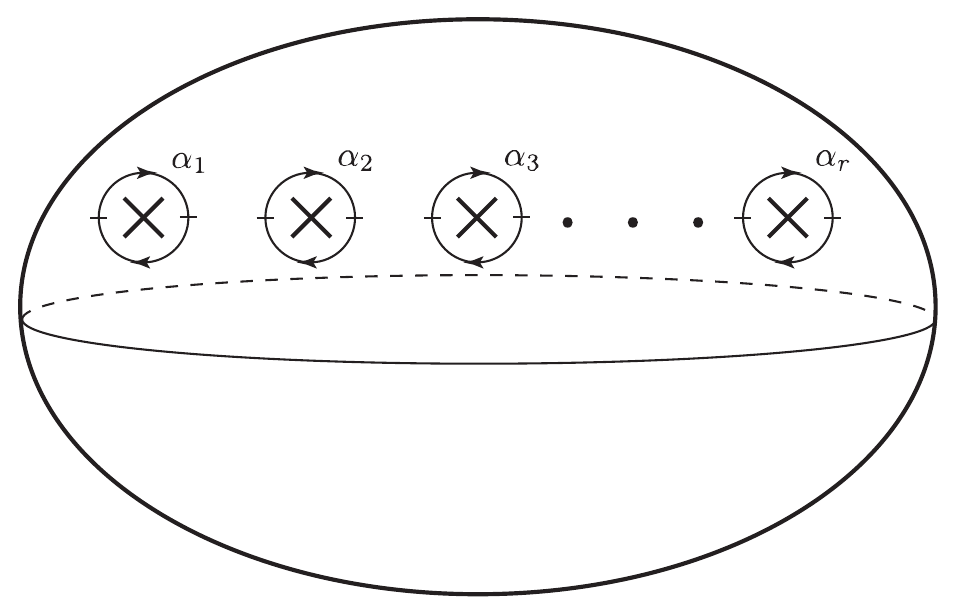}
\end{center}
\caption{\label{Nr} The space $N_r$ represented as a sphere with $r$ crosscaps.}
\end{figure}

Let $\mathcal{C}$ denote the curve shown in Figure \ref{dehntwistcurve}, which passes through the $i$th and $j$th crosscaps of $X$. Note that $\mathcal{C}$ is orientation preserving and thus has a neighborhood isomorphic to $S^1\times I$. Let $T_{i,j}$ denote the Dehn twist about $\mathcal{C}$ as defined in \cite{Lickorish}. The image of $\alpha_i$ under this map is $2\alpha_i+\alpha_j\in H_1(X;\Z/p)$, and the image of $\alpha_j$ is $-\alpha_i$. The Dehn twist $T_{i,j}$ fixes all other generators. See \cite{Pohl} for full details of this computation. 

\begin{figure}
    \centering
    \includegraphics[scale=.35]{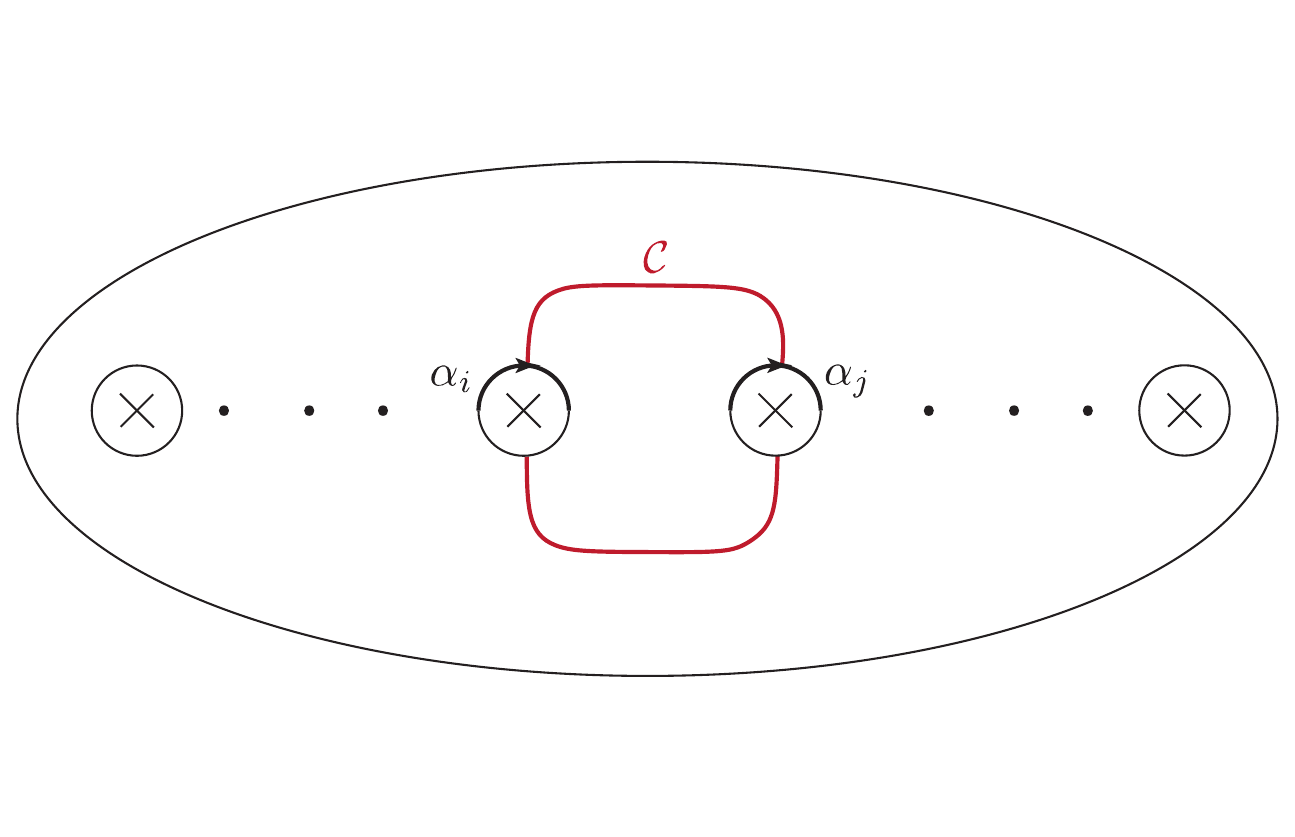}
    \caption{We perform the Dehn twist $T_{i,j}$ in the annular neighborhood of the curve $\mathcal{C}$ above.}
    \label{dehntwistcurve}
\end{figure}


For $r\geq 2$ we let $Y_{i,j}$ denote the ``crosscap slide'' map passing the $i$th crosscap through the $j$th. Note that this map is often referred to as a $Y$-homeomorphsim and is described in \cite{Lickorish} in greater detail. The image of $\alpha_i$ under this map is $-\alpha_i\in H_1(X;\Z/p)$, and the image of $\alpha_j$ is $2\alpha_i+\alpha_j$. As with the Dehn twist, all other homology generators are fixed by $Y_{i,j}$. This computation was carried out in \cite{MP85} and again in \cite{MP04} using more modern language. See also Appendix B of \cite{Pohl} for another treatment. 


The mapping class group of $\mathbb{R}P^2$ is trivial, but it turns out that for non-orientable surfaces of genus at least $2$, $\mathcal{M}(X)$ is generated by the $T_{i,j}$ and $Y_{i,j}$ for all $i\neq j$. This result is due to Chillingworth \cite{Chillingworth1969AFS}; see also \cite{SzepMCG} for a discussion using language similar to what we use here. 

\begin{proposition}\label{freenonor}
Let $X$ be a closed, connected, non-orientable surface of genus $r\geq 3$. There is only one nonzero orbit in $H^1(X;\Z/p)/\operatorname{Aut}(X)$.
\end{proposition}

 \begin{proof}

Let us start by considering the case when $r=3$. We first claim that $T_{1,2}^\ell (\alpha_1)=(\ell+1)\alpha_1 + \ell\alpha_2$, which can be quickly verified using induction. The $\ell=0$ case is immediate, and 
\begin{align*}
T_{1,2}^{\ell+1}((\ell+1)\alpha_1+\ell\alpha_2) &=(\ell+1)(2\alpha_1+\alpha_2)-\ell\alpha_1 \\
&=(2\ell +2)\alpha_1+(\ell+1)\alpha_2-\ell\alpha_1 \\
&=(\ell+2)\alpha_1+(\ell+1)\alpha_2.
\end{align*}
Let the tuple $(c_1,c_2)$ represent the element $c_1\alpha_1+c_2\alpha_2\in H_1(N_3;\Z/p)$. Now for each $1\leq k\leq p-1$, let $S_k$ be the set
\begin{align*}
S_k &=\left\{(k,0), (k\cdot 2,k), (k\cdot 3,k\cdot 2), \dots ,(k\cdot (p-1),k\cdot (p-2)),(0,k\cdot(p-1))\right\}\\
&=\left\{T_{1,2}^\ell(k,0)\mid \ell\geq 0\right\}
\end{align*}
and let $\tilde{S}_k$ be the singleton set containing $(k,k)$. Observe that $(c_1,c_2)\in S_k$ if and only if $c_1-c_2=k$. Thus every nonzero element of $H_1(X;\Z/p)$ is in at least one of the $S_k$ or $\tilde{S}_k$ for some $k$. One can also check that the map $T_{1,2}$ fixes all elements of the form $(k,k)$.  

Next we'll consider the action of $Y_{1,3}$ on the elements of $S_k$ and $\tilde{S}_k$. Since $Y_{1,3}(\alpha_1)=-\alpha_1=(p-1)\alpha_1$, the tuple $(1,0)$ maps to $(p-1,0)$. So these elements are in the same orbit, and it must be that $S_1\cup S_{p-1}$ is contained in a single orbit. 

Similarly, we have
\[Y_{1,3}(2,1)=(p-2,1)\in S_{p-3}.\] 
This implies the elements of $S_1$ and $S_{p-3}$ are in the same orbit. Therefore, $S_1\cup S_{p-1}\cup S_{p-3}$ is contained in a single orbit. Continuing in this way, we can see that in general 
\[Y_{1,3}(s,s-1)=(p-s,s-1)\in S_{p-(2s-1)}.\]
for all $1\leq s\leq p-1$. As $s$ ranges from $1$ to $p-1$, $S_{p-(2s-1)}$ ranges over all the $S_k$. This tells us that $\bigcup_{k=1}^{p-1} S_k$ is contained in a single orbit. 

Finally, we can check that $(k,k)$ must also be in this orbit for each $k$. We have
\[Y_{1,3}(k,k) = (p-k,k)\in S_{p-2k}.\]
So $\tilde{S}_k\cup S_{p-2k}$ is contained in the same orbit for each $k$. Since every nonzero element of $H_1(X;\Z/p)$ is in $S_k$ or $\tilde{S}_k$ for some $k$, there must be a single nonzero orbit in $H_1(X;\Z/p)/\text{Aut}(X)$. 

Let us now turn to the more general $r>3$ case. For ease of notation, we will denote elements of $H_1(X;\Z/p)$ by an $(r-1)$-tuple. We will show that every nonzero element is in the same orbit as $(1,0,\dots ,0)$ under the action of Dehn twists and crosscap slides. Let $(c_1,c_2,\dots ,c_{r-1})\in H_1(X;\Z/p)$ be nonzero, and let $c_i$ be the rightmost nonzero coordinate of the tuple. First suppose $i=1$. We know from the $r=3$ case that there exist compositions of $T_{1,2}$ and $Y_{1,3}$ which take $(c_1,0)$ to $(1,0)$. Since the maps $T_{j,k}$ and $Y_{j,k}$ fix all coordinates other than $j$ and $k$ of any given tuple, we can use $T_{1,2}$ and $Y_{1,r}$ to take $(c_1,0,\dots ,0)$ to $(1,0,\dots ,0)$ in the $r>3$ case. 

For $i>1$, our tuple is of the form $(c_1,c_2,\dots ,c_{i-1},c_i,0,\dots ,0)$. We again know from the $r=3$ case that there is a composition of $T_{1,2}$ and $Y_{1,3}$ which takes $(c_{i-1},c_i)$ to $(1,0)$. We can use the same compositions (replacing $Y_{1,3}$ with $Y_{1,r}$) in the $r>3$ case to take $(c_1,\dots, c_i,0,\dots ,0)$ to $(c_1,\dots, c_{i-2},1,0,\dots ,0)$. Now we have a new nonzero tuple in the same orbit as the original tuple whose rightmost nonzero coordinate is in the $(i-1)$st position. We can repeat the above process until we get that the tuple $(c_1,\dots, c_i,0,\dots,0)$ is in the same orbit as $(1,0,\dots,0)$. Since every nonzero element is in the same orbit as $(1,0,\dots,0)$, it must be that there is a single nonzero orbit in $H_1(X;\Z/p)/\text{Aut}(X)$. 
\end{proof}

Now let us go back and treat the case where $X$ is non-orientable of genus $2$. 

\begin{proposition}\label{kbottle}
There are $(p-1)/2$ nonzero orbits in $H^1(N_2;\Z/p)/\operatorname{Aut}(N_2)$.
\end{proposition}

\begin{proof}
As in the $r\geq 3$ case, we can choose to represent $N_2$ as a sphere with $2$ crosscaps $\alpha_1$ and $\alpha_2$. It is still sufficient in this case to check the action of the mapping class group on $H_1(N_2;\Z/p)\cong\langle\alpha_1\rangle\cong \Z/p$ using Dehn twists and crosscap slide maps. We again have $\alpha_1$ as a homology generator with $\alpha_2=-\alpha_1$. 


It can be easily verified that the Dehn twist about the curve passing through the two crosscaps acts trivially on $\alpha_1$ and $\alpha_2$ on homology.  



We also know that $Y_{1,2}(\alpha_1)=-\alpha_1$ and $Y_{2,1}(\alpha_1)=2\alpha_2+\alpha_1=-\alpha_1$. 

This gives us $(p-1)/2$ nonzero orbits, each containing $k\alpha_1$ and $-k\alpha_1$ for each $1\leq k\leq (p-1)/2$. 
\end{proof}


Recall that any closed, connected non-orientable surface $Y$ with free $C_p$-action must have genus $2+pr$ for some $r$. So $Y/C_p$ is a closed, connected non-orientable surface of genus $2+r$. Propositions \ref{freestuff} and \ref{freenonor} then guarantee that $Y$ must be isomorphic to $N_{2+pr}^{\text{free}}$ when $r\geq 1$ or one of the $(p-1)/2$ non-isomorphic Klein bottle actions.

\subsection{Orientable Case} 

When $X$ is an orientable surface, $\operatorname{Aut}(X)$ preserves the symplectic form given by the cup product. So the map $\operatorname{Aut}(X)\rightarrow \operatorname{Aut}(H^1(X))$ factors through the symplectic group $\operatorname{Sp}(2g,\Z/p)$. We again reference \cite{Dug19} for similar details in the $p=2$ case.

\begin{proposition}\label{forientable}
Let $X$ be a closed, connected, orientable surface of genus $g\geq 1$. There is only one nonzero orbit in $H^1(X;\Z/p)/\mathcal{M}(X)$.
\end{proposition}

\begin{proof}
We first show there is one nonzero orbit in the case $g=1$. One can easily check that the matrices $A$ and $B$ given by 
\begin{align*}
A &=\begin{pmatrix}
1 & 0 \\
1 & 1
\end{pmatrix} &
B &=\begin{pmatrix}
1 & 1 \\
0 & 1
\end{pmatrix}
\end{align*}
are in $\text{Sp}(2,\Z/p)$. For each nonzero $k\in\Z/p$, the elements of the set 
\[S_k=\left\{\begin{pmatrix} k \\
0 \end{pmatrix}, \begin{pmatrix} k \\
k \end{pmatrix}, \begin{pmatrix} k \\
2k \end{pmatrix}, \dots , \begin{pmatrix} k \\
(p-1)k \end{pmatrix}\right\}\]
are in the same orbit since $A\begin{pmatrix} k \\
nk \end{pmatrix}=\begin{pmatrix} k \\
(n+1)k \end{pmatrix}$. Similarly, for each $k$ the elements of the set 
\[T_k=\left\{\begin{pmatrix} 0 \\
k \end{pmatrix}, \begin{pmatrix} k \\
k \end{pmatrix}, \begin{pmatrix} 2k \\
k \end{pmatrix}, \dots , \begin{pmatrix} (p-1)k \\
k \end{pmatrix}\right\}\]
are in the same orbit since $B\begin{pmatrix} nk \\
k\end{pmatrix}=\begin{pmatrix} (n+1)k \\
k\end{pmatrix}$. Thus we can see that the orbit containing $\begin{pmatrix} k\\
k\end{pmatrix}$ must also contain all elements of $S_k$ and $T_k$. In particular, $S_k\cup T_k$ is contained in a single orbit for each $k$. 

For each nonzero $k\in\Z/p$, we can find its multiplicative inverse $k^{-1}$. Then $\begin{pmatrix} k \\
k^{-1} k \end{pmatrix}=\begin{pmatrix} 1\cdot k \\
1 \end{pmatrix}$ is in both $S_k$ and $T_1$. So for each $k$, the elements of $S_k$ (and thus $T_k$) are in the same orbit as $T_1$. Finally, observe that every nonzero element of $\left(\Z/p\right)^2$ is in $S_k$ or $T_k$ for some $k$. Thus, all nonzero elements are in the same orbit under the action of $\text{Sp}(2,\Z/p)$. 

Now suppose $g\geq 2$.
Choose a symplectic basis $\{e_1,f_1,\dots ,e_g,f_g\}$ so that $\langle e_i,f_i\rangle =1$, $\langle f_i,e_i\rangle =-1$, and all other pairings are 0. Denote $v\in \left(\Z/p\right)^{2g}$ by $v=[B_1,\dots ,B_g]$ where each $B_i\in \left(\Z/p\right)^2$ and 
\[v=(B_1)_1 e_1+(B_1)_2f_1 + \cdots + (B_g)_1 e_g+(B_g)_2f_g.\]
Consider the evident homomorphism
\[\text{Sp}(2,\Z/p)\times \cdots \times \text{Sp}(2,\Z/p) \rightarrow \text{Sp}(2g,\Z/p).\]
This allows us to represent orbits by vectors $[B_1,\dots ,B_g]$ with $B_i\in \{[0,0],[1,0]\}$ by the $g=1$ case. Now consider the $4\times 4$ symplectic matrix
\[A=\begin{pmatrix}
0 & I_2 \\
-I_2 & 0 
\end{pmatrix}\]
where $I_2$ is the identity matrix. Since $A$ is symplectic, so is $A^\prime =I_{2k}\oplus A \oplus I_{2g-2k-4}$ for any $0\leq k \leq g-2$. Multiplying a vector $v=[B_1,\dots ,B_g]$ by $A^\prime$ allows us to permute its $(k+1)$st and $(k+2)$nd blocks with the price of a sign. We can then multiply by the appropriate element of $\text{Sp}(2,\Z/p)\times \cdots \times \text{Sp}(2,\Z/p)$ to reduce all coefficients to $1$ or $0$. 

Thus, there are at most $g+1$ orbits of the action of $\text{Sp}(2,\Z/p)$ on $\left(\Z/p\right)^{2g}$. These orbits can be represented by the vectors
\begin{align*} 
[O,O,\dots , O] & & [T,O,\dots ,O] & & [T,T,O,\dots ,O] & & \cdots & & [T,T,\dots ,T]
\end{align*}
where $O=[0,0]$ and $T=[1,0]$. 

Let $B$ be the symplectic matrix
\[B=\begin{pmatrix}
1 & 0 & 0 & 0 \\
0 & 1 & 0 & -1 \\
1 & 0 & 1 & 0 \\
0 & 0 & 0 & 1
\end{pmatrix}\]
and observe that when $g\geq 2$, $B\oplus I_{2g-4}$ sends $[T,O,\dots ,O]$ to $[T,T,O,\dots , O]$. In particular, these two representatives are actually in the same orbit. Moreover, for $0\leq k \leq g-2$, $I_{2k}\oplus B\oplus I_{2g-2k-4}$ takes $[T,T,\dots ,T,O,\dots ,O]$ (with $T$ in the first $k+1$ entries) to the vector with $T$ in the first $k+2$ entries. Thus, all nonzero vectors in $\left(\Z/p\right)^{2g}$ are in the same orbit under the action of the symplectic group.
\end{proof}

Now let $Y$ be an orientable surface with a free $C_p$-action. We can see from Lemma \ref{euler characteristic} that the genus of $Y$ must be $1+pg$ for some $g$. This implies $Y/C_p$ is a closed, connected orientable surface of genus $1+g$. Propositions \ref{freestuff} and \ref{forientable} then imply that there is only one isomorphism class of $C_p$-spaces whose quotient by $C_p$ is $M_{1+g}$. So $Y$ must be isomorphic to $M_{1+pg}^{\text{free}}$.





\section{Non-free Classification Proof}\label{nonfreeproof}

This section contains proofs for Theorems \ref{orientableclassification} and \ref{nonorientableclassification}. In each case, we begin by establishing several lemmas describing relationships between surfaces constructed using differing equivariant surgery methods. The classification theorems are then proven using induction on the number of fixed points. 

\subsection{Proof of Classification for Orientable Surfaces}

Let us start with the orientable case.

\begin{lemma}\label{tube}
Let $X$ be a closed, connected $C_p$-surface with distinct fixed points $x$ and $y$. Then for some $i$ there exists $EB_{p,(i)}\subset X$ with $x,y\in EB_{p,(i)}$. Moreover, $\textup{nbd}(EB_{p,(i)})\subset X$ must be isomorphic to $R_{p,(i)}$ or $TR_{p,(i)}$.
\end{lemma}

\begin{proof}
This reduces to a question of how we can glue together the surfaces in Figure \ref{RpORTRp} (showing the $p=3$ case) along the red lines using equivariant maps. Any such map is completely determined by how we attach a single edge, and up to isomorphism there are only two choices. One of these produces $R_{p,(i)}$ and the other $TR_{p,(i)}$.
\end{proof}

\begin{figure}
\begin{center}
\includegraphics[scale=.6]{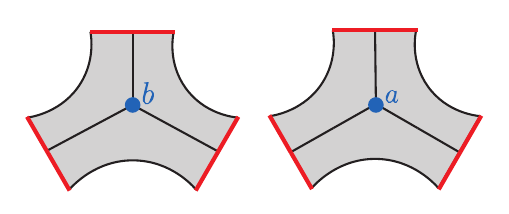}
\end{center}
\caption{\label{RpORTRp} Gluing the red edges using an equivariant map results in $R_3$ or $TR_3$.}
\end{figure}

\begin{lemma}\label{hexmap}
There is an equivariant automorphism $f$ on $TR_{p,(i)}$ with distinct fixed points $x$ and $y$ so that $f(x)=y$ and $f(y)=x$ and $f|_{\partial TR_{p,(i)}}=\textup{id}$.
\end{lemma}

\begin{proof}
Recall the polygon representation of $TR_{p,(i)}$ as shown in Figure \ref{TR3}. The action of $C_p$ on $TR_{p,(i)}$ corresponds to a rotation action by $e^{2\pi i/p}$ on the polygon. 

Let $A$ represent the annulus of width $\epsilon>0$ inside $TR_{p,(i)}$ so that $\partial TR_{p,(i)}$ is a boundary component of $A$. Define $f$ so that $f|_A$ is the Dehn twist with $f|_{\partial TR_{p,(i)}}=\textup{id}$ and $f$ restricted to the other boundary component of $A$ is given by $180^\circ$ rotation. Then let $f|_{TR_{p,(i)}\setminus A}$ act as rotation by $180^\circ$. Notice that $f$ respects the $C_p$-action of $TR_{p,(i)}$ and swaps $x$ and $y$ as desired.
\end{proof}

\begin{lemma}\label{M1[3]+[TAT]}
If $x,y\in \Hex_1$ are distinct fixed points, then $\Hex_1+_x[TR_p]\cong \Hex_1+_y[TR_p]$. Moreover,
\[\Hex_1+_x[TR_p]\cong \Sph_{p-1}[4].\]
\end{lemma}

\begin{proof}
Given any two distinct fixed points $x,y\in\Hex_1$, there is a copy of $TR_p$ containing them. By Lemma \ref{hexmap}, there is an automorphism $\tilde{\varphi}$ of $TR_p\subset \Hex_1$ swapping $x$ and $y$. This can be extended to an automorphism $\varphi$ of $\Hex_1$ by defining $\varphi$ to be $\tilde{\varphi}$ on $TR_p$ and the identity everwhere else. Thus we can define an isomorphism $\Hex_1+_x[TR_p]\rightarrow\Hex_1+_y[TR_p]$ given by $\varphi$ everywhere outside of the added copy of $TR_p$. 

Observe that $\Hex_1+_x[TR_p]$ can be obtained by taking two copies of $TR_p$ and identifying their boundaries. Figure \ref{M1lemma} shows how this gives us $\Sph_{p-1}[4]$ in the $p=3$ case. Choose one copy of $TR_p$ to be a neighborhood of the red $EB_p$. It's complement in $\Sph_{p-1}[4]$ is another copy of $TR_p$ containing the purple $EB_p$.
\end{proof}

\begin{lemma}\label{Bn+[TR3]}
If $x,y\in \Hex_n$ ($n\geq 2$) are distinct fixed points, then 
\[\Hex_n+_x[TR_p]\cong \Hex_n+_y[TR_p].\] 
In other words, twisted ribbon surgery on $\Hex_n$ is independent of the fixed point chosen.
\end{lemma}

The proof is almost identical to that of Lemma \ref{M1[3]+[TAT]}. The idea is that any two fixed points in $\Hex_n$ are contained in a copy of $TR_p$. More specifically, this argument shows that $\Hex_n+_x[TR_p]\cong \left(\Hex_{n-1}+(k+2)[R_p]\right)\#_pM_g$. 

\begin{figure}
\begin{center}
\includegraphics[scale=.3]{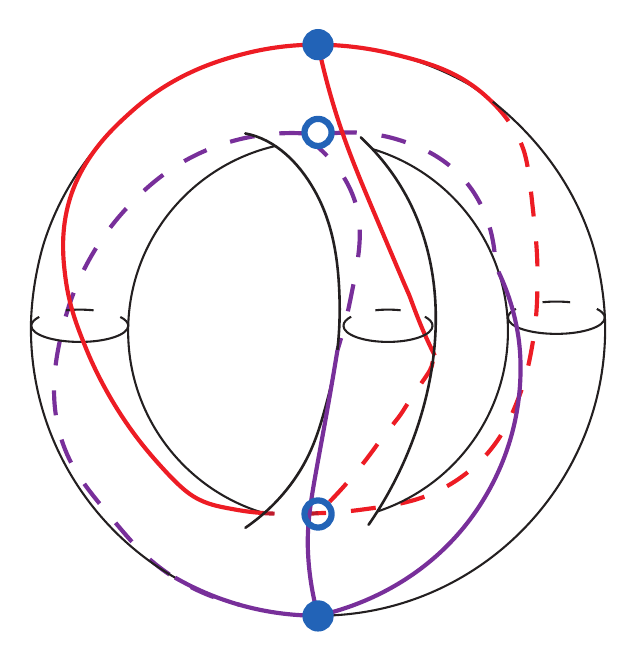}
\end{center}
\caption{\label{M1lemma} Two copies of $TR_3$ inside $\Sph_2[4]$.}
\end{figure}

\begin{lemma}\label{M2k+[TAT]}
If $x$ and $y$ are distinct fixed points in $\Sph_{(p-1)k+pg}[2k+2]$ for some $k,g\geq 0$, then $\Sph_{(p-1)k+pg}[2k+2]+_x[TR_p]\cong \Sph_{(p-1)k+pg}[2k+2]+_y[TR_p]$. In other words, twisted ribbon surgery on $\Sph_{(p-q)k+pg}[2k+2]$ is independent of the chosen fixed point.
\end{lemma}

\begin{proof}
We can choose to represent $\Sph_{(p-1)k+pg}[2k+2]$ in the following way: 
\begin{enumerate}
\item Start with $S^{2,1}$. 
\item Choose $k+1$ disks $D_1,\dots D_{k+1}$ centered at the equator of $S^{2,1}$ so that $\sigma^s D_i\cap \sigma^{s^\prime}D_j=\emptyset$ for all $i,j,s,s^\prime$.
\item Perform $\#_pM_g$-surgery using $D_{k+1}$ and its conjugates.
\item Remove $D_1,\dots ,D_k$ and their conjugates to perform $+[R_p]$-surgery $k$ times. Let ${R_p}_i$ denote the copy of $R_p$ glued to the boundary of $D_i\cup\sigma D_i\cup\cdots \cup \sigma^{p-1} D_i$.
\end{enumerate} 
Suppose each copy of $R_p$ is glued onto $S^{2,1}$ as shown in Figure \ref{ATgluing}. We call $a$ the ``north pole'' of $R_p$ and $b$ the ``south pole''. Figures \ref{M2kTAT} and \ref{M2kTAT2} depict a path $\alpha$ (in green) from the north pole of ${R_p}_i$ for some $i$ to the north pole of $S^{2,1}$ or ${R_p}_j$ for some $j$. This figure only shows the path in the case where $k=2$ and $g=0$, but in all other cases a similar path can be chosen. Observe that a neighborhood of $\alpha\cup \sigma \alpha\cup\cdots\cup \sigma^{p-1} \alpha$ is isomorphic to $TR_p$. This can be verified by checking that this neighborhood has only a single boundary component. In this case, we know there exists an automorphism of $\Sph_{(p-1)k+pg}[2k+2]$ swapping the two north poles. Similarly, if given two south poles we can find a copy of $TR_p$ containing them. Thus if $x$ and $y$ are both north poles (respectively south poles), then $\Sph_{(p-1)k+pg}[2k+2]+_x[TR_p]\cong \Sph_{(p-1)k+pg}[2k+2]+_y[TR_p]$. 

\begin{figure}
\begin{center}
\includegraphics[scale=.5]{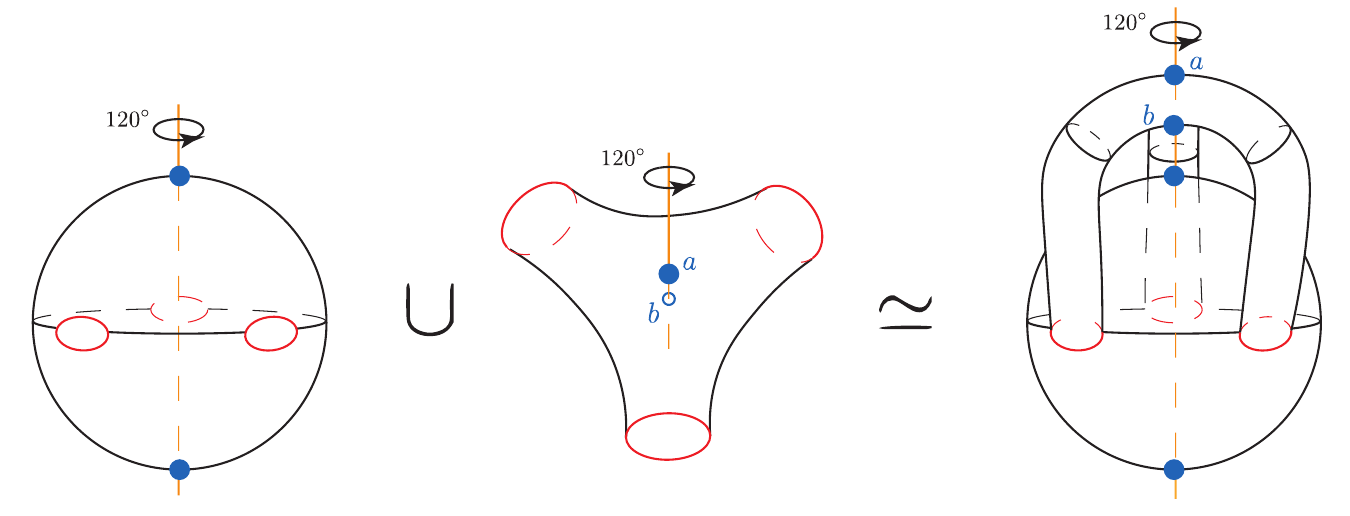}
\end{center}
\caption{\label{ATgluing} A representation of $\Sph_2[4]$ using $+[R_3]$-surgery on $S^{2,1}$.}
\end{figure}

\begin{figure}
\begin{center}
\includegraphics[scale=.33]{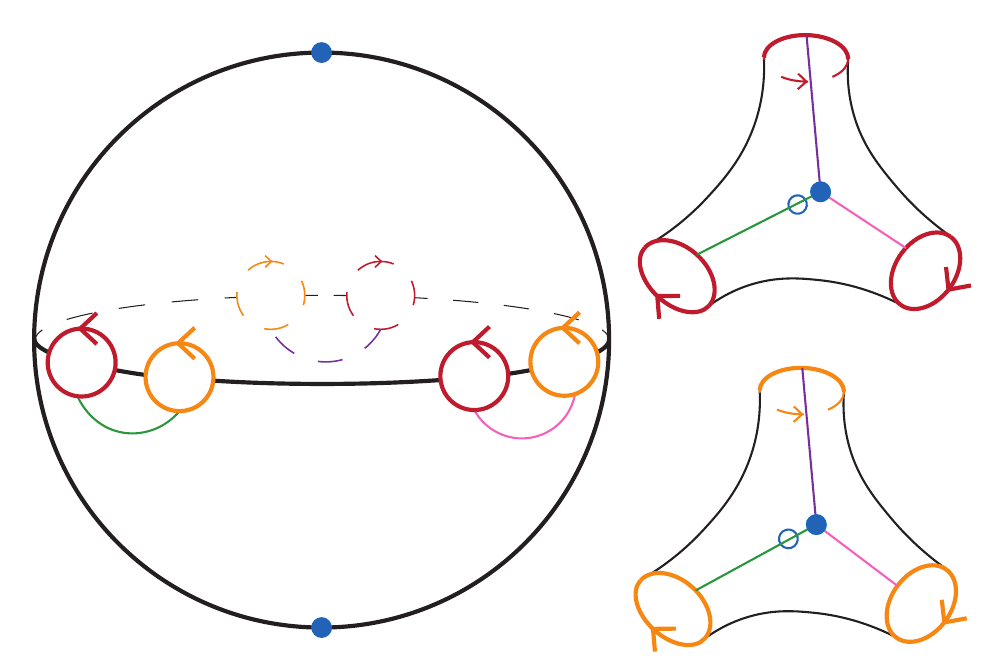}
\end{center}
\caption{\label{M2kTAT}A neighborhood of the copy of $EB$ shown here is isomorphic to $TR_3$.}
\end{figure}

\begin{figure}
\begin{center}
\includegraphics[scale=.33]{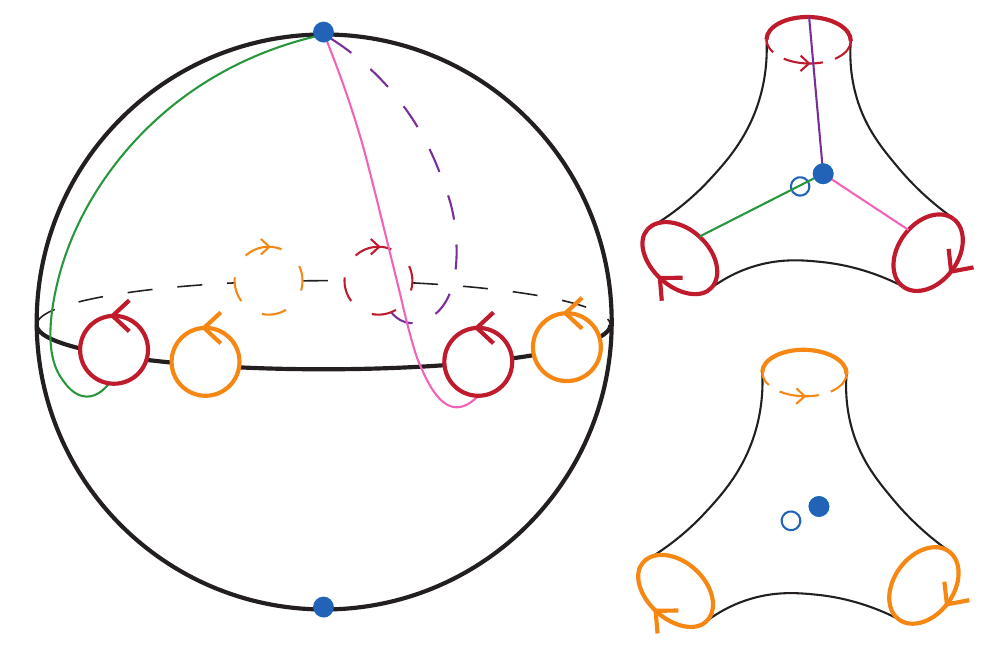}
\end{center}
\caption{\label{M2kTAT2} A neighborhood of the copy of $EB$ shown here is isomorphic to $TR_3$.}
\end{figure}

It remains to show that if $x$ is a north pole and $y$ is a south pole, then twisted ribbon surgery on $\Sph_{(p-1)k+pg}[2k+2]+_x[TR_p]$ at the points $x$ and $y$ result in isomorphic spaces. We will show this by considering the case $x=a$ and $y=b^\prime$ as depicted in Figures \ref{M2kiso} and \ref{M2kiso2}. The argument for cases when $k>1$ or $g>0$ are similar. If we can show the isomorphism in this case, then for any north pole $x^\prime$ and any south pole and $y^\prime$, we have
\begin{align*}
\Sph_{(p-1)k+pg}[2k+2]+_{x^\prime}[TR_p]&\cong \Sph_{(p-1)k+pg}[2k+2]+_x[TR_p] \\
&\cong \Sph_{(p-1)k+pg}[2k+2]+_y[TR_p] \\
&\cong \Sph_{(p-1)k+pg}[2k+2]+_{y^\prime}[TR_p].
\end{align*}
Figure \ref{M2kiso} depicts the result of $+_a[TR_3]$-surgery on $M_2[4]$, and Figure \ref{M2kiso2} shows $M_2[4]+_{b^\prime}[TR_3]$. We can construct an isomorphism between these spaces as reflection through the plane of the hexagon. 

\begin{figure}
\begin{center}
\includegraphics[scale=.45]{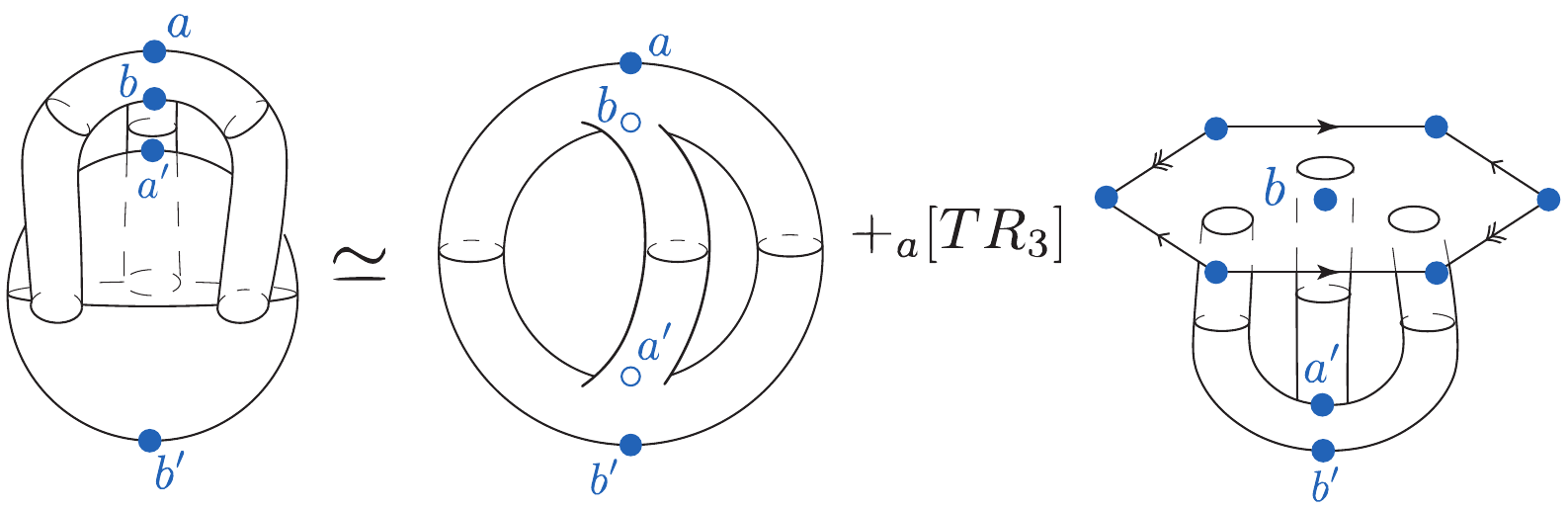}
\end{center}
\caption{\label{M2kiso} The result of $+_a[TR_3]$-surgery on $\Sph_{2k+3g}[2k+2]$ for $k=1$, $g=0$.}
\end{figure}

\begin{figure}
\begin{center}
\includegraphics[scale=.45]{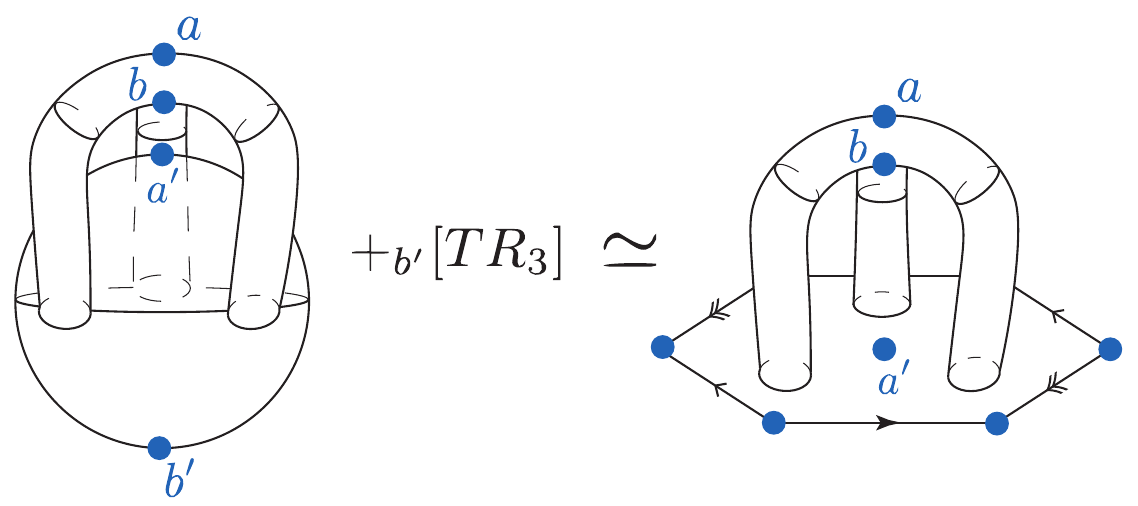}
\end{center}
\caption{\label{M2kiso2} The result of $+_{b^\prime}[TR_3]$-surgery on $\Sph_{2k+3g}[2k+2]$ for $k=1$, $g=0$.}
\end{figure}
\end{proof}

So far we have proven that $X+_x[TR_p]$ is independent of $x$ when $X$ is of the form $\Hex_n\#_p M_g$ or $S^{2,1}+[R_p]\#_p M_g$. We will now spend some time understanding when twisted ribbon surgery fails to be independent of its chosen fixed point. 


\begin{proposition}\label{hex2}
There does not exist an equivariant isomorphism between the $C_p$-spaces $\Sph_{2(p-1)}[6]$ and $\Hex_2$ (even up to the action of $\operatorname{Aut}(C_p)$).
\end{proposition}

\begin{proof}
Let $X$ be a nontrivial, orientable $C_p$-space, and let $X^{C_p}$ denote the fixed set of $X$. We start by defining a map $X^{C_p}\rightarrow C_p$. Fix an orientation for $X$, and consider the induced orientation on $X/C_p$. For each fixed point $x\in X^{C_p}$, let $\bar{x}$ represent the image of $x$ in $X/C_p$. Choose a small loop going around $\bar{x}$ in the direction of the chosen orientation. We can then lift this loop to a path in $X$ going from a point $y$ to $gy$ for some $g\in C_p$. Note that the element $g$ is independent of choice for $y$. In this way, we can define the map $X^{C_p}\rightarrow C_p$ given by $x\mapsto g$. 

Theorem 1.1 of \cite{Ding} states that this map determines the $C_p$-space $X$ up to isomorphism. 

Let us now turn our attention to $\Hex_2$ and $\Sph_{2(p-1)}[6]$. We will show by direct computation that the maps $\Hex_2^{C_p}\rightarrow C_p$ and $\Sph_{2(p-1)}[6]^{C_p}\rightarrow C_p$ as defined above must be distinct.  

We focus our attention on the $p=3$ case since the argument can be extended to all odd primes. Let $g$ be the generator of $C_3$ corresponding to counter-clockwise rotation of $\Hex_2$ by $120^\circ$ about the axis passing through the center of the hexagons. Figure \ref{class1} demonstrates that for any fixed point $x\in\Hex_2$, the image of $x$ under the above map is $g$. 

To see this, start by labeling the six fixed points of $\Hex_2$ as $x_1,\dots ,x_6$. Next choose an orientation for $\Hex_2$ and consider the induced orientation on $\Hex_2/C_3\simeq S^2$. Figure \ref{class1} depicts $\Hex_2$ (left) and $S^2$ (right) with the chosen orientation in gray. We can then choose a loop in the direction of the orientation about the image of each $x_i$ in $S^2$. Each of these loops can be lifted to some path in $\Hex_2$. Let $\tilde{x}_i$ ($1\leq i\leq 6$) denote the starting point of the path lifted from the $i$th loop. Figure \ref{class1} demonstrates that for each $i$, we get a path from $\tilde{x}_i$ to $g\tilde{x}_i$. 

For example, the green loop on the right of Figure \ref{class1} goes about the fixed point $x_1$. We can lift it to the green path in $\Hex_2$. This path starts at the point labeled $\tilde{x}_1$ and ends at the image of $\tilde{x}_1$ under the action of $g$. So our map in this case sends $x_1$ to $g$. Since $\Hex_2^{C_3}$ just consists of the fixed points $x_1,x_2,\dots ,x_6$, we can describe the above map as the tuple $(g,g,g,g,g,g)$. 

Let us now choose to represent the space $\Sph_4[6]$ as depicted on the left of Figure \ref{class2}. Let $g$ represent counter-clockwise rotation of $\Sph_{4}[6]$ by $120^\circ$ about the axis passing through the center of the hexagons, and label the six fixed points as $x_1,x_2,\dots ,x_6$. We can then fix an orientation for $\Sph_4[6]$ and choose oriented paths about the image of $x_i$ in $\Sph_4[6]/C_3\simeq S^2$ for each $i$. As before, we lift each of these loops to a path starting at the point $\tilde{x}_i$, and we look at the endpoint of each lifted path. Figure \ref{class2} demonstrates that these endpoints are $g\tilde{x}_1$, $g\tilde{x}_2$, $g\tilde{x}_3$, $g^2\tilde{x}_4$, $g^2\tilde{x}_5$, and $g^2\tilde{x}_6$. Another way to represent this map $\Sph_4[6]^{C_3}\rightarrow C_3$ is with the tuple $(g,g,g,g^2,g^2,g^2)$. 

\begin{figure}
\begin{center}
\includegraphics[scale=.5]{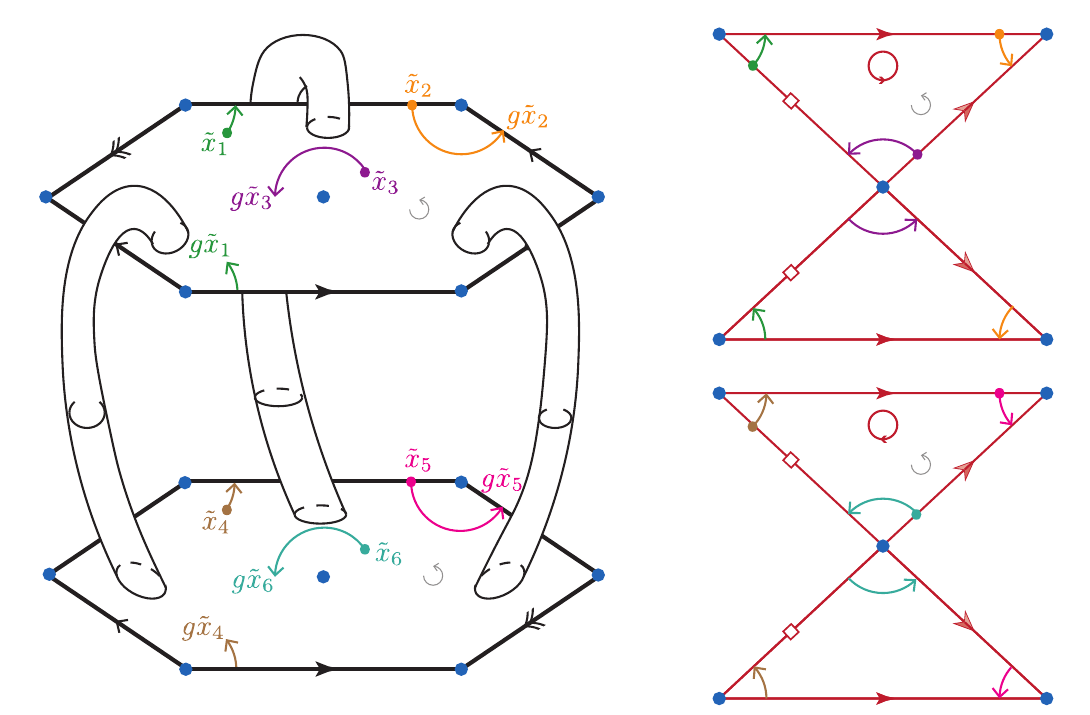}
\end{center}
\caption{\label{class1} The map $\Hex_2^{C_3}\rightarrow C_3$, described by the tuple $(g,g,g,g,g,g)$.}
\end{figure}

\begin{figure}
\begin{center}
\includegraphics[scale=.5]{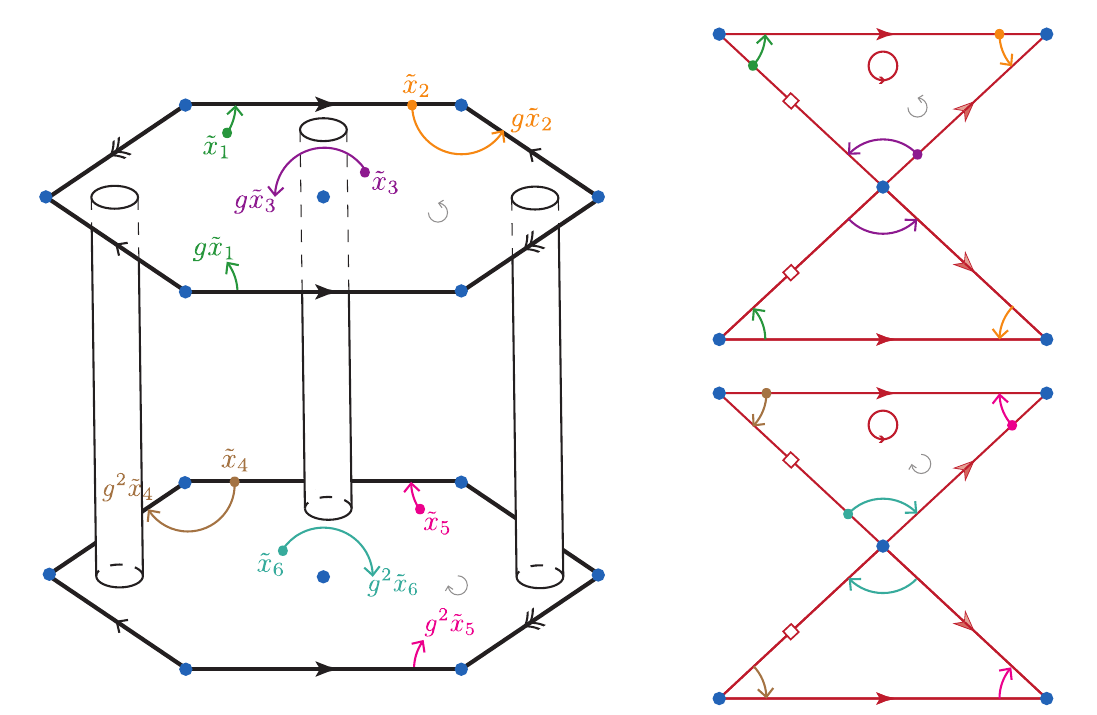}
\end{center}
\caption{\label{class2} The map $\Sph_4[6]^{C_3}\rightarrow C_3$, described by the tuple $(g,g,g,g^2,g^2,g^2)$.}
\end{figure}

Even up to a relabeling of the fixed points and an action of $\operatorname{Aut}(C_3)$, the maps described by $(g,g,g,g,g,g)$ and $(g,g,g,g^2,g^2,g^2)$ must be distinct. In other words, it cannot be the case that $\Sph_4[6]$ and $\Hex_2$ are isomorphic.
\end{proof}

More generally, the same argument shows that $\Hex_2+k[R_p]\#_p M_g$ is not isomorphic to $\Sph_4[6]+k[R_p]\#_p M_g$ for any $k$,$g$. 

\begin{remark}
The same methods can also be used to show $\Hex_{n_1}+k_1[R_p]\#_p M_{g_1}$, $\Hex_{n_2}+k_2[R_p]\#_p M_{g_2}$, and $S^{2,1}+k_3[R_p]\#_p M_{g_3}$ are always in distinct isomorphism classes (unless of course $n_1=n_2$, $k_1=k_2$, and $g_1=g_2$). 
\end{remark}

\begin{lemma}\label{twoTR3isoclases}
When $k\geq 1$, there are two isomorphism classes of $C_p$-spaces of the form $\Hex_{1,(p-1)/2+(p-1)k+pg}[3+2k]+_x[TR_p]$ which depend on the choice of fixed point $x$. In particular, given a fixed point $x$, $\Hex_{1,(p-1)/2+(p-1)k+pg}[3+2k]+_x[TR_p]$ is isomorphic to one of the following:
\begin{enumerate}
\item $\Sph_{(p-1)(k+1)+pg}[2+2(k+1)]$
\item $\left(\Hex_2+(k-1)[R_p]\right)\#_p M_g$
\end{enumerate}
\end{lemma}

\begin{proof}
We can represent $\Hex_{1,(p-1)/2+(p-1)k+pg}[3+2k]$ by choosing $k+1$ disks $D_1,\dots D_{k+1}$ on $\Hex_1$ so that $\sigma^s D_i\cap \sigma^{s^\prime}D_j=\emptyset$ for all $i,j,s,s^\prime$. Remove each $\sigma^jD_i$. Then attach a copy of $R_p$ (denoted ${R_p}_i$) to $\partial D_i\cup \partial\left(\sigma D_i\right)\cup\cdots\cup \partial \left(\sigma^{p-1} D_i\right)$ for each $i=1,\dots ,k$. Then attach a copy of $C_p\times \left(M_g\setminus D^2\right)$ to $\partial D_{k+1}\cup\partial \left(\sigma D_{k+1}\right)\cup\cdots\cup \partial \left(\sigma^{p-1} D_{k+1}\right)$. For simplicity of notation, we will let $X$ denote the space $\Hex_{1,(p-1)/2+(p-1)k+pg}[3+2k]$ for the remainder of the proof. 

A similar argument as in the previous case shows that if $x$ and $y$ are the north poles (respectively south poles) of ${R_p}_i$ and ${R_p}_j$ for some $i,j$, then we can find a copy of $TR_p$ containing $x$ and $y$. This implies that $X+_x[TR_p]\cong X+_y[TR_p]$ for all such $x$ and $y$. 

Let $a,b,c$ be the fixed points originating from the copy of $\Hex_1$ as depicted in Figure \ref{M3TAT}. This figure depicts a copy of $EB_p$ in $X$ containing the north pole of $(R_p)_1$ and $c$ with a neighborhood isomorphic to $TR_p$. Figure \ref{M3TAT} depicts the case $k=1,g=0$, but one could construct a similar copy of $EB_p$ in all other cases. Recall additionally from Lemma \ref{M1[3]+[TAT]} that there is a copy of $TR_p$ containing $a$ and $c$ as well as a copy containing $b$ and $c$. So we have that $X+_x[TR_p]\cong X+_y[TR_p]$ when $x,y\in \{\text{north pole of }(R_p)_i\mid 1\leq i\leq k\}\cup \{a,b,c\}$. This also holds if $x,y\in\{\text{south pole of }(R_p)_i\mid 1\leq i \leq k\}$. 

At this point we have demonstrated there are at most two isomorphism classes of $\Hex_{1,(p-1)/2+(p-1)k+pg}[3+2k]+_x[TR_p]$. We know from Lemma \ref{M1[3]+[TAT]} that $X+_c[TR_p]\cong \Sph_{p-1}[4]+k[R_p]\#_p M_g$. By construction in Example \ref{Bn}, we also know that $X+_x[TR_p]\cong \Hex_2+(k-1)[R_p]\#_p M_g$ when $x\in\{\text{south pole of }(R_p)_i\mid 1\leq i \leq k\}$.  

We know from Proposition \ref{hex2} and subsequent remarks that these spaces are not isomorphic. So there must be exactly two isomorphism classes of spaces of the form $\left(\Hex_{1,(p-1)/2+(p-1)k+pg}[3+2k]\right)+_?[TR_p]$.

\begin{figure}
\begin{center}
\includegraphics[scale=.6]{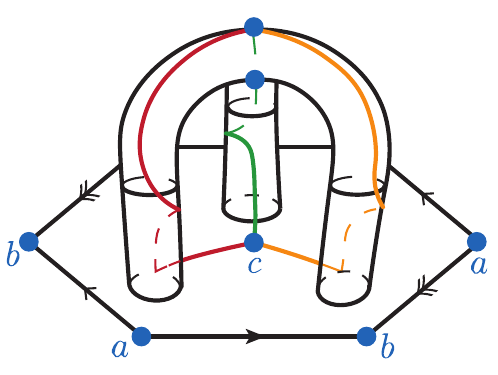}
\end{center}
\caption{\label{M3TAT} A copy of $EB_3$ whose neighborhood is isomorphic to $TR_3$.}
\end{figure}
\end{proof}

\begin{corollary}\label{twoTR3isoclasses2}
For $n\geq 2$ and $k\geq 1$, there are two isomorphism classes of $C_p$-spaces of the form $\left(\Hex_n+k[R_p]\right)\#_p M_g+_x[TR_p]$ which depend on the choice of fixed point $x$. Specifically, given a fixed point $x$, $\left(\Hex_n+k[R_p]\right)\#_p M_g+_x[TR_p]$ is isomorphic to one of the following:
\begin{enumerate}
\item $\left(\Hex_{n+1}+(k-1)[R_p]\right)\#_p M_g$
\item $\left(\Hex_{n-1}+(k+2)[R_p]\right)\#_p M_g$
\end{enumerate}
\end{corollary}

The same ideas presented in the proof of Lemma \ref{twoTR3isoclases} can be extended to this more general case. 

Finally, we present a lemma which will help prove the inductive step of our main classification theorem.

\begin{lemma}\label{connectedness}
Let $X$ be a connected $C_p$-surface for which $X-[R_p]$ is defined. If $F(X)\geq 3$, then $X-[R_p]$ must also be connected. 
\end{lemma}

\begin{proof}
Fix a copy of $R_p\subset X$ on which we will perform $-[R_p]$ surgery. Since $F(X)\geq 3$, there exists at least one additional fixed point $x\in X$ such that $x\not\in R_p$. In order to show that $X-[R_p]$ is connected, it suffices to show that $X\setminus R_p$ (the space obtained by removing $R_p$ from $X$ but before gluing in the $p$ conjugate disks) is connected. 

We first claim that given any point $y$ in the boundary of $X\setminus R_p$, there is a path from the fixed point $x$ to $y$. First note that the connected component of $X\setminus R_p$ containing $x$ must have at least one boundary component (which we will call $C$). Otherwise, $X$ could not have been connected. Thus there is a path from $x$ to any point on $C$. A conjugate to any such path would be a path from $x$ to $\sigma^i C$. Thus, $x$ must be in the same connected component as each boundary component of $X\setminus R_p$. 

Since $X$ is connected, every point $z\in X\setminus R_p$ must be in the same connected component as at least one boundary component. Thus every point in $X\setminus R_p$ must lie in a single boundary component.
\end{proof}


We are now ready to revisit Theorem \ref{orientableclassification} and provide a proof of the result.

\begin{theorem}
Let $X$ be a connected, closed, orientable surface with an action of $C_p$. Then $X$ can be constructed via one of the following surgery procedures, up to $\operatorname{Aut}(C_p)$ actions on each of the pieces.
\begin{enumerate}
	\item $M_{1+pg}^{\text{free}}:= M_1^{\text{free}}\#_pM_g$, $g\geq 0$
	\item $\Sph_{(p-1)k+pg}[2k+2]:=\left(S^{2,1}+k[R_p]\right)\#_pM_g$, $k,g\geq 0$
	\item $\Hex_{n,(3n-2)(p-1)/2+(p-1)k+pg}[3n+2k]:=\left(\Hex_n+k[R_p]\right) \#_pM_g$, $k,g\geq 0$, $n\geq 1$
\end{enumerate}
\end{theorem}

\begin{proof}
We induct on the number of fixed points $F$. 

First let $X$ be a free orientable space. By the classification of free $C_p$-spaces done in Section \ref{freeproof}, $X\cong M_{1+pg}^{\text{free}}$ for some $g\geq 0$.  

The case where $X$ is orientable and $F=1$ does not occur. A proof of this fact can be found in Example 3.3 of \cite{Bro91} or Theorem 7.1 of \cite{AB67}. Let us move on to the case $F=2$. Let $x,y\in X$ be distinct fixed points. By Lemma \ref{tube}, there exists $R_p\subset X$ or $TR_p\subset X$ containing $x$ and $y$. The latter case is not possible since $X-[TR_p]$ would be a closed, orientable $C_p$-surface with a single fixed point. So $x$ and $y$ are contained in some $R_p$ in $X$. Then by the $F=0$ case, $X-[R_p]\cong M_{1+pg}^{\text{free}}$ or $X-[R_p]\cong M_g\times C_p$. We know from Figure \ref{-Rpchoices} that $+[R_p]$ surgery on either of these spaces results in $S^{2,1}\#_pM_{g^\prime}$ for some $g^\prime\geq 0$. Thus $X$ must be isomorphic to $\Sph_{pg}[2]$. 

We additionally observe in the case $F=2$ that since $X\cong S^{2,1}\#_pM_{g^\prime}$, there is an equivariant automorphism of $X$ swapping the fixed points $x$ and $y$. This map $\varphi$ can be defined as a reflection through the plane perpendicular to the axis of rotation which bisects $X$. We can thus define an isomorphism $X+_y[TR_p]\rightarrow X+_x[TR_p]$ given by $\varphi$ everywhere outside of the added copy of $TR_p$.  

Next assume $F=3$. Again, we can find distinct fixed points $x$ and $y$ in $X$ which are contained in $R_p\subset X$ or $TR_p\subset X$. The former is impossible since $X-[R_p]$ would be a closed, orientable $C_p$-surface with one fixed point. Thus, $x,y\in TR_p$ in $X$. So $X-_{x,y}[TR_p]\cong S^{2,1}\#_pM_g$ for some $g$ by the previous $F=2$ case. Finally we can observe that $\left(S^{2,1}\#_pM_g\right)+[TR_p]\cong \Hex_1\#_pM_g$. Since $S^{2,1}\#_pM_g+_?[TR_p]$ is independent of the chosen fixed point, we can conclude that $X\cong \Hex_1\#_pM_g$.  

Since $X\cong \Hex_1\#_pM_g$, all three fixed points of $X$ live in a neighborhood isomorphic to $\Hex_1\setminus\left(D_2\times C_p\right)$. So given any two fixed points in $X$, there exists $TR_p\subset X$ containing them. By Lemma \ref{hexmap} we can construct an equivariant automorphism of $X$ swapping any two of its fixed points. Therefore $+[TR_p]$ surgery on $X$ is invariant of the choice of fixed point. 

For the inductive hypothesis, let $3 < \ell$. For any $\ell^\prime$ with $3\leq \ell^\prime <\ell$, suppose that (1) if $Z$ is a connected, closed, orientable $C_p$-surface with $F=\ell^\prime$, then $Z$ is isomorphic to $\Sph_{(p-1)k+pg}[2k+2]$ or $\Hex_{n,(3n-2)(p-1)/2+(p-1)k}[3n+2k]$ for some $k,g\geq 0$ and $n\geq 1$, and (2) if $x$ and $y$ in $Z$ are distinct fixed points, then $Z+_x[TR_p]\cong Z+_y[TR_p]$. Now let $X$ be a closed, orientable $C_p$-surface with $F=\ell$. Let $x,y\in X$ be distinct fixed points. By Lemma \ref{tube}, there exists $R_p\subset X$ or $TR_p\subset X$ containing $x$ and $y$. 

Suppose first that $x$ and $y$ are contained in $R_p\subset X$. Then $X-[R_p]$ has $\ell-2\geq 2$ fixed points. Since $X$ was connected and $X-[R_p]$ has at least one fixed point, $X-[R_p]$ must also be connected by Lemma \ref{connectedness}. So we can invoke the inductive hypothesis to conclude that $X-[R_p]$ is isomorphic to one of the following:
\begin{enumerate}
\item $\Sph_{(p-1)k+pg}[2k+2]\cong\left(S^{2,1}+k[R_p]\right)\#_pM_g$
\item $\left(\Hex_n+k[R_p]\right)\#_p M_g$.
\end{enumerate}
In the first case, we can conclude 
\[X\cong \left(S^{2,1}+(k+1)[R_p]\right)\#_pM_g\cong \Sph_{(p-1)(k+1)+pg}[2(k+1)+2].\]
In the second case, 
it follows that
\[X\cong \left(\Hex_n+(k+1)[R_p]\right)\#_p M_g.\] 

If $x$ and $y$ are contained in $TR_p\subset X$, then $X-_{x,y}[TR_p]$ has $\ell-1\geq 3$ fixed points. By the inductive hypothesis, $X-_{x,y}[TR_p]$ is isomorphic to one of the following:
\begin{enumerate}
\item $\Sph_{(p-1)k+pg}[2k+2]\cong\left(S^{2,1}+k[R_p]\right)\#_pM_g$ for some $k\geq 1$ and $g\geq 0$
\item $\left(\Hex_n+k[R_p]\right)\#_p M_g$ for some $n\geq 1$ and $k,g\geq 0$.
\end{enumerate}
We know from Lemma \ref{M2k+[TAT]} that $+_?[TR_p]$-surgery on $\Sph_{(p-1)k+pg}[2k+2]$ is independent of the chosen fixed point. So if $X-_{x,y}[TR_p]\cong \Sph_{(p-1)k+pg}[2k+2]$, then 
\[X\cong \left(\Hex_1+k[R_p]\right)\#_pM_g.\]

Next suppose $X-_{x,y}[TR_p]\cong \left(\Hex_n+k[R_p]\right)\#_p M_g$ for some $n\geq 1$ and $k,g\geq 0$. Again, we know from Corollary \ref{twoTR3isoclasses2} that if $k\geq 1$ there are two isomorphism classes of spaces for $\left(\left(\Hex_n+k[R_p]\right)\#_p M_g\right)+_a[TR_p]$, depending on the choice of fixed point $a$. In one case we have
\[X\cong  \left(\Hex_{n-1}+(k+2)[R_p]\right)\#_pM_g.\]
This is also the result of $+_a[TR_p]$-surgery on $X$ when $k=0$.
Assuming $k\geq 1$, it is also possible that
\[X\cong \left(\Hex_{n+1}+(k-1)[R_p]\right)\#_pM_g.\]  
\end{proof}


For the remainder of this section, we use $\tilde{N}_n$ to denote the space $N_n\setminus D^2$. 

\subsection{Free Actions on Non-orientable Surfaces with Boundary}

Our next goal is to prove the classification theorem for non-orientable $C_p$-surfaces. We saw that there were no orientable $C_p$-surfaces with a single fixed point, but this is not the case for non-orientable surfaces. In order to prove the $F=1$ case of our classification theorem, we need to lay a bit of ground work. We start with a treatment of free $C_p$-actions on $\tilde{N}_{pn+1}$ for $n\geq 0$.  


\begin{proposition}\label{free with boundary}
Up to the action of $\operatorname{Aut}(C_p)$ there is a single isomorphism class of free $C_p$ actions on $\tilde{N}_{pn+1}$ for all $n\geq 0$. 
\end{proposition}

More precisely, there are $(p-1)/2$ non-isomorphic actions on $\tilde{N}_{pn+1}$. These are the $\operatorname{Aut}(C_p)$-conjugates of $MB_p\#_p N_n$ where $MB_p$ is defined in Section \ref{surgintro}. 

Given a free $C_p$ action on $\tilde{N}_{pn+1}$, the quotient $\tilde{N}_{pn+1}/C_p$ must be a non-orientable surface with a single boundary component and Euler characteristic $\frac{1}{p}(1-(pn+1))=-n$. The only such space is $\tilde{N}_{n+1}$. 

Recall from Section \ref{surgintro} that $\mathcal{S}(\tilde{N}_{n+1})$ denotes the set of isomorphism classes of path-connected, free $C_p$ spaces $X$ so that $X/C_p\cong \tilde{N}_{n+1}$. There is a bijection between $\mathcal{S}(\tilde{N}_{n+1})$ and the set of nonzero orbits of $H^1(\tilde{N}_{n+1};\Z/p)$ under the action of $\operatorname{Aut}(\tilde{N}_{n+1})$. 

To prove Proposition \ref{free with boundary}, we will consider three cases: $n=0$, $n=1$, and $n>1$. When $n>1$, we will show that there are at most $(p+1)/2$ nonzero orbits in $H_{1}(\tilde{N}_{n+1};\Z/p)/\operatorname{Aut}(\tilde{N}_{n+1})$. Then we construct a free $C_p$ space $Y\not\cong \tilde{N}_{pn+1}$ of genus $pn+1$ whose quotient by $C_p$ is $\tilde{N}_{n+1}$. This will guarantee that the $\frac{p-1}{2}$ conjugate actions of $C_p$ on $\tilde{N}_{pn+1}$ coming from $MB_p\#_pN_n$ can be the only such actions. We carry out a similar procedure in the $n=1$ case, instead showing that there are $p-1$ nonzero orbits in $H_{1}(\tilde{N}_{2};\Z/p)/\operatorname{Aut}(\tilde{N}_{2})$ and constructing a non-equivariant space distinct from $\tilde{N}_{n+1}$ with $(p-1)/2$ non-isomorphic free $C_p$-actions. The $n=0$ case will prove to be even simpler, with only $(p-1)/2$ nonzero orbits in $H_1(\tilde{N}_1)/\operatorname{Aut}(\tilde{N}_1)$.

\begin{figure}
\begin{center}
\includegraphics[scale=.35]{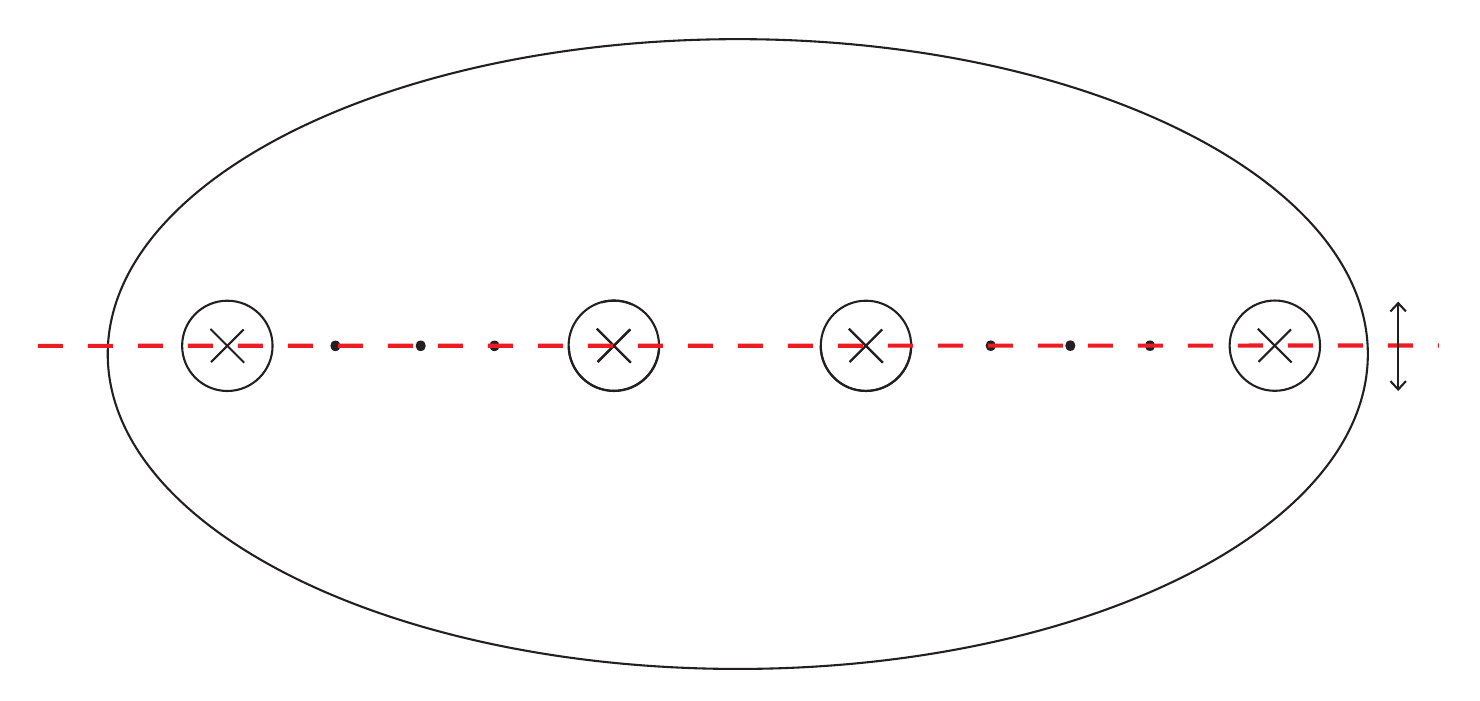}
\end{center}
\caption{\label{reflection} Reflection about the orange line sends every $\alpha_i$ to $-\alpha_i$.}
\end{figure}

\begin{proof}
Our proof will be very reminiscent of that of Theorem \ref{freenonor}. Represent $\tilde{N}_{n+1}$ as a disk with $n+1$ crosscaps, and pick a basis $\{\alpha_1,\dots , \alpha_{n+1}\}$ for $H_1(\tilde{N}_{n+1};\Z/p)=\left(\Z/p\right)^{n+1}$ given by the center circles of the crosscaps. 
Recall our notation $T_{i,j}$ for the Dehn twist about the curve passing through the $i$th and $j$th crosscaps and $Y_{i,j}$ for the crosscap slide which passes the $i$th crosscap through the $j$th. In addition to these mapping class group elements, let $\psi$ denote the reflection as shown in Figure \ref{reflection}. 


When $n=0$, we do not have Dehn twists or crosscap slide homeomorphisms. It is quick to check that $\psi$ sends $k\alpha_1$ to $(p-k)\alpha_1$. This gives us at most $(p-1)/2$ nonzero orbits in $H_1(\tilde{N}_1;\Z/p)/\operatorname{Aut}(\tilde{N}_1)$. In fact we can conclude that there are exactly $(p-1)/2$ nonzero orbits because there should also be \emph{at least} $(p-1)/2$ nonzero orbits corresponding to the $(p-1)/2$ non-isomorphic free actions on the M\"{o}bius band defined in Section \ref{surgintro}.

Skip the $n=1$ case for now, let us assume $n\geq 2$. Let $\mathbf{c}=(c_1,\dots ,c_{n+1})$ be a nonzero element of $H_1(\tilde{N}_{n+1};\Z/p)$. We will first show that there are at most $p-1$ nontrivial orbits with representatives of the form
$(k,0,\dots,0)$ for some $1\leq k\leq \frac{p-1}{2}$ or $(\ell,\ell,\dots,\ell)$ for some $1\leq \ell\leq \frac{p-1}{2}$. Let $c_i$ be the rightmost nonzero entry of $\mathbf{c}$ with the property that $c_i\neq c_{i-1}$. We first claim there exists some power of $T_{i-1,i}$ so that
\[(c_1,\dots ,c_{i-1},c_i,\dots,c_{n+1})\sim (c_1,\dots ,c_{i-1}-c_i,0,c_{i+1},\dots ,c_{n+1}).\]
We showed in the proof of Proposition \ref{freenonor} that applying $T_{i-1,i}$ to the tuple $s$ times produces the tuple whose $(i-1)$st coordinate is $(s+1)c_{i-1}-sc_i$ and whose $i$th coordinate is $,sc_i-(s-1)c_{i-1}$. Since $c_{i-1}\neq c_i$, there exists some positive integer $s$ so that $sc_{i-1}-(s-1)c_i\equiv 0\pmod{p}$. For such an $s$, it is therefore also true that $(s+1)c_{i-1}-sc_i\equiv c_{i-1}-c_i\pmod{p}$. So applying $T_{i-1,i}^s$ to the tuple $(c_1,\dots ,c_{n+1})$ produces
\[(c_1,\dots ,c_{i-1}-c_i,0,c_{i+1},\dots ,c_{n+1})\]
as desired. Notice that applying the appropriate power of $T_{i-1,i}$ either increases the number of zeros in the tuple (in the case that $c_{i-1}\neq 0$) or shifts an existing zero to the right one position (in the case $c_{i-1}=0$). Repeat this process to obtain the orbit representative $(c_1,\dots ,c_{n+1})\sim (k,k,\dots,k,0,\dots,0)$ with $k\neq 0$ and $1\leq \ell\leq n+1$ nonzero entries.

When $\ell<n+1$, 
\[Y_{\ell,n+1}(k,k,\dots,k,0,\dots ,0)=(k,\dots,k,p-k,0,\dots,0)\]
Since the $\ell$th entry is not equal to the $(\ell-1)$st entry, we can repeat the steps outlined in the previous paragraph until we obtain
\[(c_1,\dots,c_{n+1})\sim (k^\prime,\dots,k^\prime,0,\dots,0)\]
with $k^\prime\neq 0$ and $\ell^\prime<\ell$ nonzero entries. Since $k$ is nonzero and $k\neq p-k$, we know the number of zeros will strictly increase with this process. Therefore we can repeat it until
\[(c_1,\dots,c_n)\sim (k,0,\dots,0)\]
for some nonzero $k$. 

In the case that $\ell=n+1$ we have $c_i=c_1$ for all $i$. So 
\[(c_1,\dots ,c_{n+1})= (c_1,c_1,\dots,c_1).\] 
Note that the action of $\psi$ puts $(k,0,\dots,0)$ in the same orbit as $(p-k,0,\dots ,0)$ and similarly puts $(k,k,\dots,k)$ in the same orbit as $(p-k,p-k,\dots,p-k)$
, giving us at most $p-1$ nonzero orbits. 

To finish the $n\geq 2$ case, we will now check that all elements of the form $(k,0,\dots ,0)$ are in the same orbit under the action of Dehn twists and crosscap slides. Let $(1,a,0\dots,0)$ be an element of $H_1(\tilde{N}_{n+1};\Z/p)$ with $a\neq 0$. Note that since $n\geq 2$, this element at least one zero entry. Based on our previous arguments, we know this is in the same orbit as $(1-a,0,\dots,0)$ and $(a-1,0,\dots,0)$. Alternatively, we can see that under the action of $Y_{2,3}$ followed by several Dehn twists, $(1,a,0,\dots,0)$ is in the same orbit as $(1,-a,0,\dots,0)$ and $(1+a,0,\dots,0)$. Putting all of this together, we are able to conclude that $(a+1,0,\dots,0)$ and $(a-1,0,\dots,0)$ are in the same orbit for all $a$. this is enough to conclude that all elements of the form $(k,0,\dots,0)$ are in the same orbit when $k\neq 0$. As desired, this leaves us with $(p+1)/2$ nontrivial orbits with representatives $(1,0,\dots,0)$ and $(\ell,\ell,\dots,\ell)$ for all $1\leq \ell\leq (p-1)/2$.

When $n=1$, $H_1(\tilde{N}_2;\Z/p)=\Z/p\oplus\Z/p$. As with the $n>1$ case, analyzing the action of Dehn twists and crosscap slides on the homology generators gives us at most $p-1$ nontrivial orbits with representatives of the form $(k,0)$ for $1\leq k \leq p-1$ and $(\ell,\ell)$ for $1\leq \ell\leq p-1$. We will see below that these must represent $p-1$ distinct orbits in $H_1(\tilde{N}_2)/\operatorname{Aut}(\tilde{N}_2)$.

Let $Y$ be the space obtained by removing $p$ conjugate disks from $N_2^{\text{free}}\#_pN_{n-1}$. As proved in Section \ref{freenonor}, this space has $(p-1)/2$ non-isomorphic free $C_p$-actions when $n=1$ and just one action when $n>1$. The quotient of $Y$ by any of its free actions is $\tilde{N}_{n+1}$ as desired. Moreover, $Y\not\cong \tilde{N}_{pn+1}$ since these spaces do not have the same number of boundary components. 
\end{proof}


As in the last section, the trivial orbit of $H_1(\tilde{N}_{n+1};\Z/p)$ corresponds to the non-path-connected $C_p$-space $C_p\times \tilde{N}_{n+1}$.

\subsection{Proof of Classification for Non-orientable Surfaces} 


\begin{lemma}\label{N1[1]+[AT]}
There is an equivariant isomorphism
\[\Hex_1\#_p N_1\cong N_1[1]+[R_p]\]
\end{lemma}

\begin{proof}
We will prove this result for the case $p=3$, noting that $p>3$ is similar. Figure \ref{surg2} shows us how $\left(\Hex_1\#_3N_1\right)-[R_3]\cong N_1[1]$. To begin, we represent $\Hex_1\#_3N_1$ as our usual hexagon picture with fixed points $a$, $b$, and $c$ as well as $3$ crosscaps. A copy of $EB$ containing $a$ and $b$ can be seen in red in the figure on the left. One can check that a tubular neighborhood of this $EB$ has three boundary components and thus must be isomorphic to $R_3$. The middle of Figure \ref{surg2} shows the result of removing this copy of $R_3$. To complete $-[R_3]$ surgery, we glue in the orange, pink, and green disks along the resulting boundary. To more easily see these identifications, we can first perform the intermediate step of ``flipping'' the red regions and identifying the yellow edges, then having the red change back to grey. The third picture on the right shows the result of the completed $-[R_3]$ surgery. The resulting space is isomorphic to $N_1[1]$. The original statement then follows from Lemma \ref{-[Rp] invariance}. 
\end{proof}

\begin{figure}
\begin{center}
\includegraphics[scale=.4]{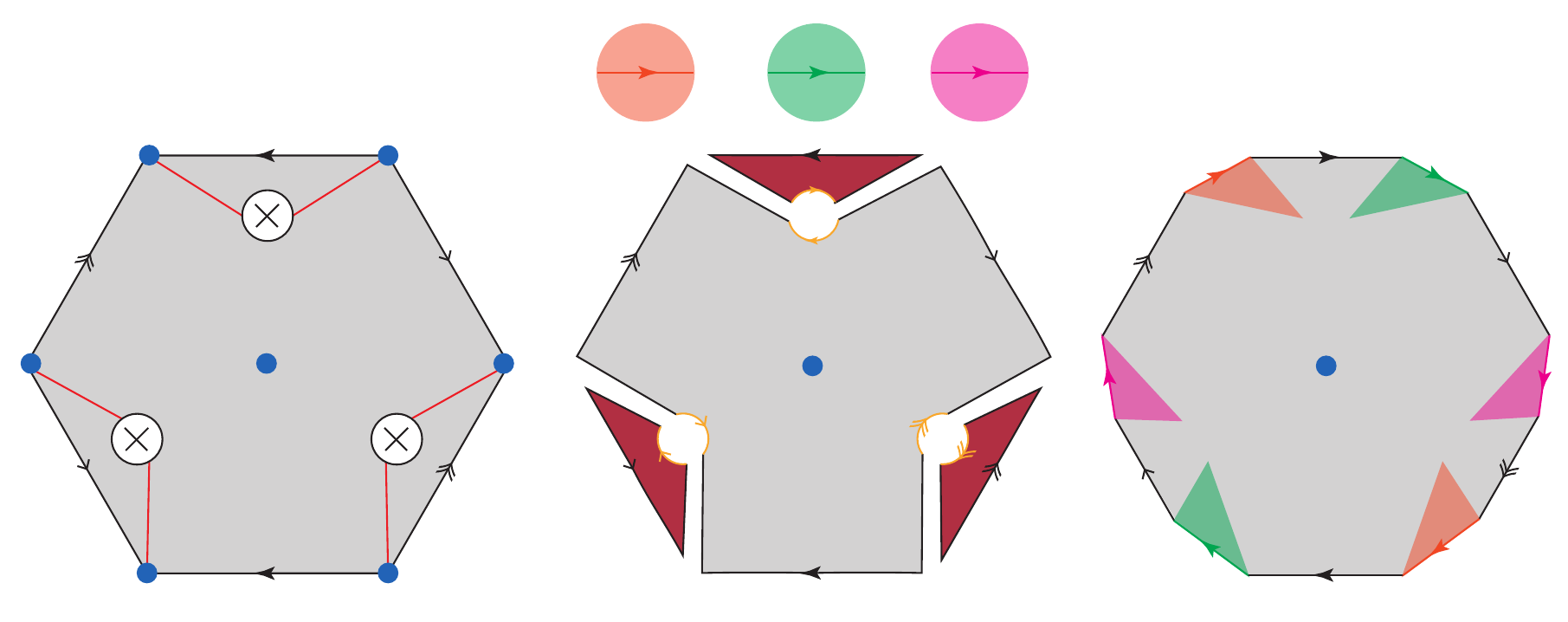}
\end{center}
\caption{\label{surg2} The procedure $\left(\Hex_1\#_p N_1\right)-[R_p]$.}
\end{figure}

\begin{lemma}\label{N2free+[AT]}
There is an equivariant isomorphism
\[N_2^{\text{free}}+[R_p]\cong S^{2,1}\#_pN_2\]
\end{lemma}

\begin{proof}
If we perform $-[R_p]$ surgery on a neighborhood of the copy of $EB$ from $S^{2,1}\#_pN_2$ shown in Figure \ref{surg3}, the result is a connected, non-orientable surface with a free $C_p$-action. Since $-[R_p]$-surgery reduces $\beta$-genus by $2(p-1)$, this surface must have genus $\beta=2$. In particular, it must be $N_2^{\text{free}}$ by our classification of free $C_p$ spaces. It follows from Lemma \ref{-[Rp] invariance} that $N_2^{\text{free}}+[R_p]\cong S^{2,1}\#_pN_2$.
\end{proof}

\begin{figure}
\begin{center}
\includegraphics[scale=.4]{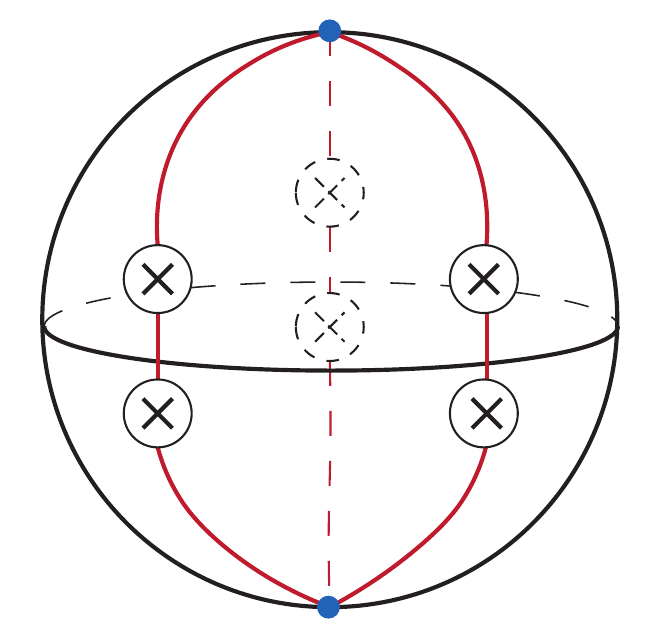}
\end{center}
\caption{\label{surg3} A copy of $EB$ whose neighborhood is $R_3$.}
\end{figure}

\begin{lemma}\label{N1[1]+[TAT]}
There is an equivariant isomorphism
\[N_1[1]+[TR_p]\cong S^{2,1}\#_pN_1\]
\end{lemma}

\begin{proof}
The $C_p$-space $\Hex_1+[FMB_p]$ can be constructed in two ways. In addition to performing $+[FMB_p]$ surgery on $\Hex_1$, we could start by constructing $N_1[1]$ as $S^{2,1}+[FMB_p]$. We can then build $N_1[1]+[TR_p]$ by performing the $+[TR_p]$ surgery on the remaining fixed point. These two constructions are demonstrated in Figure \ref{N1[1]+[TR3]}. 

Since both of these constructions yield the same space, it follows that $N_1[1]+[TR_p]\cong \Hex_1+[FMB_p]$. If we next remove a copy of $R_p$ from $\Hex_1+[FMB_p]$ as shown in Figure \ref{surg4}, the result is $C_p\times MB$ where $MB$ denotes the M\"{o}bius band. Thus, when we finish the $-[R_p]$ surgery on $\Hex_1+[FMB_p]$ by gluing in $p$ disks on the boundary components, this leaves us with $C_p\times N_1$. Since $C_p\times N_1\cong \left(S^{2,1}\#_pN_1\right)-[R_p]$, we get that $\Hex_1+[FMB_p]\cong S^{2,1}\#_pN_1$ by Lemma \ref{-[Rp] invariance}.
\end{proof}

\begin{figure}
\begin{center}
\includegraphics[scale=.5]{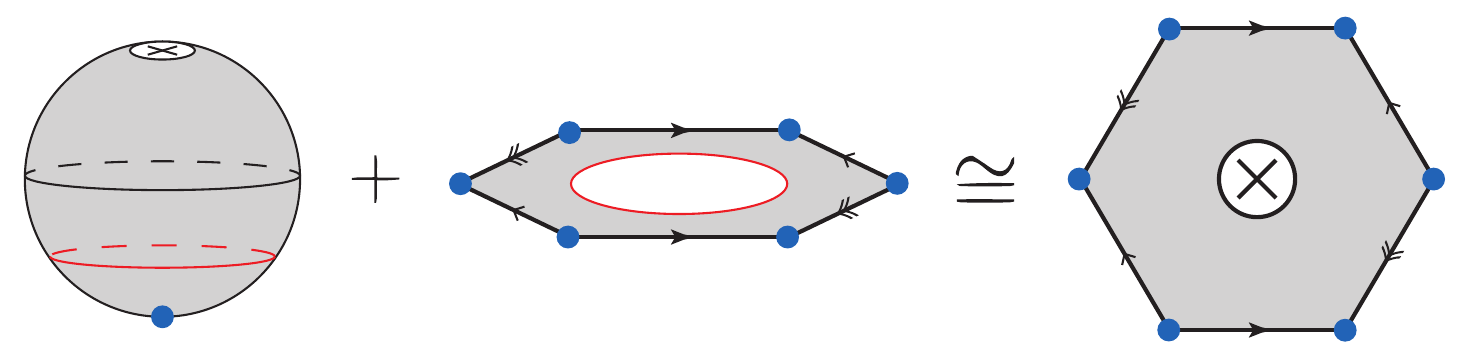}
\end{center}
\caption{\label{N1[1]+[TR3]} The equivariant surgery procedure $N_1[1]+[TR_3]$.}
\end{figure}

\begin{figure}
\begin{center}
\includegraphics[scale=.65]{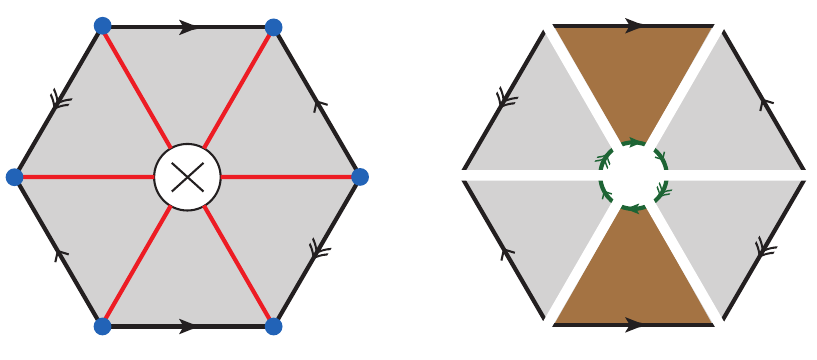}
\end{center}
\caption{\label{surg4} Removing $R_3$ from $\Hex_1$ results in $MB\times C_3$.}
\end{figure}

We are now ready to restate and prove Theorem \ref{nonfreenonorientable} for the classification of non-orientable $C_p$-surfaces.

\begin{theorem}
Let $X$ be a connected, closed, non-orientable surface with an action of $C_p$. Then $X$ can be constructed via one of the following surgery procedures, up to $\operatorname{Aut}(C_p)$ actions on each of the pieces.
\begin{enumerate}
	\item $N_{2+pr}^{\text{free}}\cong N_2^{\text{free}}\#_pN_r$, $r\geq 0$
	\item $N_{2(p-1)k+pr}[2k+2]\cong\left(S^{2,1}+k[R_p]\right)\#_pN_r$, $r\geq 1$
	\item $N_{1+2(p-1)k+pr}[1+2k]\cong\left(N_1[1]+k[R_p]\right)\#_pN_r$, $k,r\geq 0$
\end{enumerate}
Moreover, the space $X$ is determined by $F$ and $\beta$, with the condition that $F\equiv 2-\beta\pmod{p}$.
\end{theorem}

\begin{proof}
We induct on the number of fixed points $F$ of $X$. 

First let $X$ be a free non-orientable space. By the classification of free $C_p$-spaces, $X\cong N_{2+pr}^{\text{free}}$ for some $r\geq 0$.  

Let $X$ be a connected, closed, non-orientable $C_p$-surface with $F=1$. Then $X$ must have genus $pr+1$ for some $r\geq 0$ by Lemma \ref{euler characteristic}. Suppose $Y$ is another closed, connected, genus $pr+1$ non-orientable $C_p$-surface with a single fixed point. Let $\tilde{X}$ (respectively $\tilde{Y}$) denote the $C_p$-space $X\setminus D^{2,1}$ (respectively $Y\setminus D^{2,1}$) where $D^{2,1}$ is a neighborhood of the fixed point of $X$ (respectively $Y$). Recall that $\tilde{N}_{pr+1}$ has $(p-1)/2$ non-trivial, non-isomorphic $C_p$ actions up to isomorphism by Proposition \ref{free with boundary}. After altering the action on $Y$ by $\operatorname{Aut}(C_p)$, we can make the action on $\partial\tilde{Y}$ match that on $\partial\tilde{X}$. Then $\tilde{X}\cong \tilde{Y}$, which extends to an equivariant isomorphism $X\rightarrow Y$. Thus there is only one non-orientable $C_p$-surface of genus $pr+1$ with $F=1$, so it must be isomorphic to $N_1[1]\#_pN_r$. 

Suppose $F=2$. Let $x$ and $y$ be the two distinct fixed points of $X$. By Lemma \ref{tube}, there exists $R_p\subset X$ or $TR_p\subset X$ containing $x$ and $y$. If there exists $R_p\subset X$ containing $x$ and $y$, then $X-[R_p]$ is a free, non-orientable $C_p$-space. If $X-[R_p]$ is connected, then $X-[R_p]\cong N_{2+pr}^{\text{free}}$ for some $r\geq 0$. So $X\cong N_{2+pr}^{\text{free}}+[R_p]\cong S^{2,1}\#_pN_{r+2}$ by Lemma \ref{N2free+[AT]}. If $X-[R_p]$ is not connected, then it must be isomorphic to $N_{r^\prime}\times C_p$ for some $r^\prime\geq 1$. In this case, we can see that $X\cong S^{2,1}\#_pN_{r^\prime}$.  

Suppose instead we find that $x$ and $y$ are contained in some $TR_p\subset X$. Then $X-_{x,y}[TR_p]$ is a closed, connected, non-orientable $C_p$-surface with $1$ fixed point. In particular, $X-_{x,y}[TR_p]\cong N_1[1]\#_pN_r$ for some $r$ by what we already showed. Recall that equivariant connected sum surgery commutes with all types of $C_p$-ribbon surgeries. Since $X$ is the result of $+[TR_p]$-surgery on $N_1[1]\#_pN_r$, Lemma \ref{N1[1]+[TAT]} tells us that 
\[X\cong \left(N_1[1]+[TR_p]\right)\#_pN_r \cong \left(S^{2,1}\#_pN_1\right)\#_pN_r\cong S^{2,1}\#_pN_{r+1}.\]

We next claim that for a closed, non-orientable $C_p$-surface with $F=2$, there exists a path $\alpha$ between the two fixed points so that a neighborhood of $\alpha\cup\sigma \alpha\cup\cdots \cup \sigma^{p-1}\alpha$ is isomorphic to $TR_p$. We just showed that $X\cong S^{2,1}\#_pN_r$ for some $r\geq 1$, so we can represent $X$ by a copy of $S^{2,1}$ with $pr$ crosscaps at the equator. Figure \ref{n3rtat} shows a path $\alpha$ on $X$ with the desired property in the case when $r=2$ and $p=3$. By Lemma \ref{hexmap}, there exists an automorphism of $X$ swapping its fixed points. As in previous cases, this allows us to conclude that $+[TR_p]$ surgery on $X$ is independent of the chosen fixed point. 

\begin{figure}
\begin{center}
\includegraphics[scale=.3]{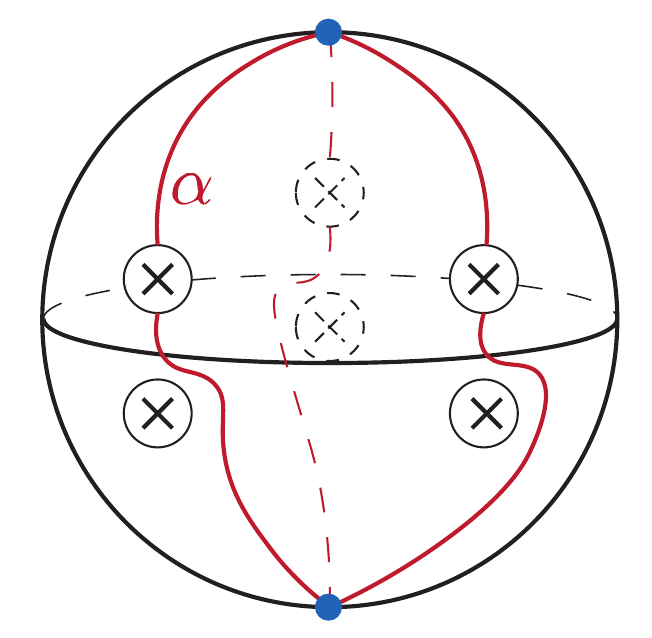}
\end{center}
\caption{\label{n3rtat} A choice of $\alpha$ whose conjugates have a neighborhood isomorphic to $TR_3$.}
\end{figure}

For the inductive hypothesis, let $2<\ell$. For any $\ell^\prime$ with $2\leq \ell^\prime<\ell$, suppose that (1) if $A$ is a connected, closed, non-orientable $C_p$-surface with $F=\ell^\prime$, then $Z$ is isomorphic to $N_{2(p-1)k+pr}[2k+2]$ or $N_{1+2(p-1)k+pr}[1+2k]$ for some $k,r$, and (2) if $x$ and $y$ in $Z$ are distinct fixed points, then $Z+_x[TR_p]\cong Z+_y[TR_p]$.  Now let $X$ be a closed, non-orientable $C_p$-surface with $F=\ell$. Let $x,y\in X$ be distinct fixed points. By Lemma \ref{tube}, there exists $R_p\subset X$ or $TR_p\subset X$ containing $x$ and $y$. 

Suppose first that $x$ and $y$ are contained in $R_p\subset X$. Then $X-[R_p]$ has $\ell-2\geq 1$ fixed points and is thus connected by Lemma \ref{connectedness}. By the inductive hypothesis, $X-[R_p]$ is isomorphic to one of the following:
\begin{enumerate}
\item $N_{2(p-1)k+pr}[2k+2]\cong\left(S^{2,1}+k[R_p]\right)\#_pN_r$
\item $N_{1+2(p-1)k+pr}[1+2k]\cong\left(N_1[1]+k[R_p]\right)\#_pN_r$.
\end{enumerate}
In the first case, we can conclude 
\[X\cong \left(S^{2,1}+(k+1)[R_p]\right)\#_pN_r\cong N_{2(p-1)(k+1)+pr}[2(k+1)+2].\]
In the second case, we have
\[X\cong \left(N_1[1]+(k+1)[R_p]\right)\#_pN_r\cong N_{1+2(p-1)(k+1)+pr}[1+2(k+1)].\] 

If $x$ and $y$ are contained in $TR_p\subset X$, then $X-[TR_p]$ has $\ell-1\geq 2$ fixed points. By the inductive hypothesis, $X-[TR_p]$ is isomorphic to one of the following:
\begin{enumerate}
\item $N_{2(p-1)k+pr}[2k+2]\cong\left(S^{2,1}+k[R_p]\right)\#_pN_r$ ($r\geq 1$)
\item $N_{1+2(p-1)k+pr}[1+2k]\cong\left(N_1[1]+k[R_p]\right)\#_pN_r$.
\end{enumerate}
We also know from the inductive assumption that $+[TR_p]$-surgery on $X-[TR_p]$ is independent of the chosen fixed point, so $\left(X-[TR_p]\right)+[TR_p]\cong X$. Thus in the first case, we can choose to center our $+[TR_p]$ surgery on the north pole of $S^{2,1}$. Since $r\geq 1$, we have
\begin{align*}
X &\cong \left(\left(\Hex_1\#_pN_1\right)+k[R_p]\right)\#_pN_{r-1} \\
&\cong \left(N_1[1]+(k+1)[R_p]\right)\#_pN_{r-1} \\
&\cong N_{1+2(p-1)(k+1)+p(r-1)}[1+2(k+1)]
\end{align*}
where the second isomorphism is by Lemma \ref{N1[1]+[AT]} and the first isomorphism follows from the commutativity of $+[R_p]$-surgery and equivariant connected sum surgery. In the second case, we can choose to center our $+[TR_p]$ surgery on the fixed point originating from the copy of $N_1[1]$. By Lemma \ref{N1[1]+[TAT]}, we get 
\[X\cong \left(S^{2,1}+k[R_p]\right)\#_pN_{r+1}\cong N_{2(p-1)k+p(r+1)}[2k+2].\] 

Next we will show that if $x$ and $y$ are distinct fixed points in $X$, then $X+_x[TR_p]\cong X+_y[TR_p]$. The case where $X\cong N_{2(p-1)k+pr}[2k+2]$ is nearly identical to the orientable case $\Sph_{(p-1)k+pg}[2k+2]$, so we will provide the proof of $+[TR_p]$ invariance only for $X\cong N_{1+2(p-1)k+pr}[1+2k]$. 

We represent $N_{1+2(p-1)k+pr}[1+2k]$ by first choosing a disk $D$ in $N_1[1]$ that does not intersect its conjugates. Then choose a representation of $N_{2(p-1)(k-1)+pr}[2(k-1)+2]$ using the same construction as for $\Sph_{(p-1)k+pg}[2k+2]$ in Lemma \ref{M2k+[TAT]}. Next remove $p$ disjoint conjugate disks $D^\prime ,\sigma D^\prime, \dots ,\sigma^{p-1}D^\prime$ from the equator of the sphere $S^{2,1}$ used to construct $N_{2(p-1)(k-1)+pr}[2(k-1)+2]$. Remove $D$ and its conjugates from $N_1[1]$ and identify $\partial \sigma^i D$ with $\partial \sigma^i D^\prime$ (renaming $D^\prime$ if necessary). Let $c$ denote the fixed point in $N_1[1]$. We will show that for any other fixed point $x$ there exists an equivariant automorphism of $N_{1+2(p-1)k+pr}[1+2k]$ which exchanges $x$ and $c$. If we can show this, then composition of these automorphisms allows us to swap any two fixed points in $N_{2(p-1)(k-1)+pr}[2(k-1)+2]$. 

Let $x\neq c$ be a fixed point in $N_{1+2(p-1)k+pr}[1+2k]$. Then $x$ is either contained in the copy of $S^{2,1}$ or $(R_p)_i$ for some $i$. In any case, there exists a path $\alpha$ from $x$ to $c$ with a neighborhood of $\alpha\cup\sigma \alpha\cup \cdots \cup\sigma^{p-1}\alpha$ isomorphic to $TR_p$. Figure \ref{n5tat} shows how to construct such a path $\alpha$ when $x\in S^{2,1}$. Note that this figure does not show the $(R_p)_i$, but $\alpha$ can be constructed so that it does not intersect any of the $(R_p)_i$. Similarly, Figure \ref{n9tat} shows how to construct $\alpha$ when $x\in (R_p)_i$ for some $i$. The choice of $\alpha$ is similar for all $i$. Again note that $\alpha$ can be constructed so that it does not intersect $(R_p)_j$ when $j\neq i$. One can check that the paths depicted in these figures have a neighborhood isomorphic to $TR_p$ by checking that the chosen neighborhood has a single boundary component. Since $x$ and $c$ are contained in a copy of $TR_p\subset N_{1+2(p-1)k+pr}[1+2k]$, we know from Lemma \ref{hexmap} that there exists an automorphism of $N_{1+2(p-1)k+pr}[1+2k]$ swapping $x$ and $c$. The result then follows from induction.

\begin{figure}
\begin{minipage}{0.4\textwidth}
\begin{center}
\includegraphics[scale=.3]{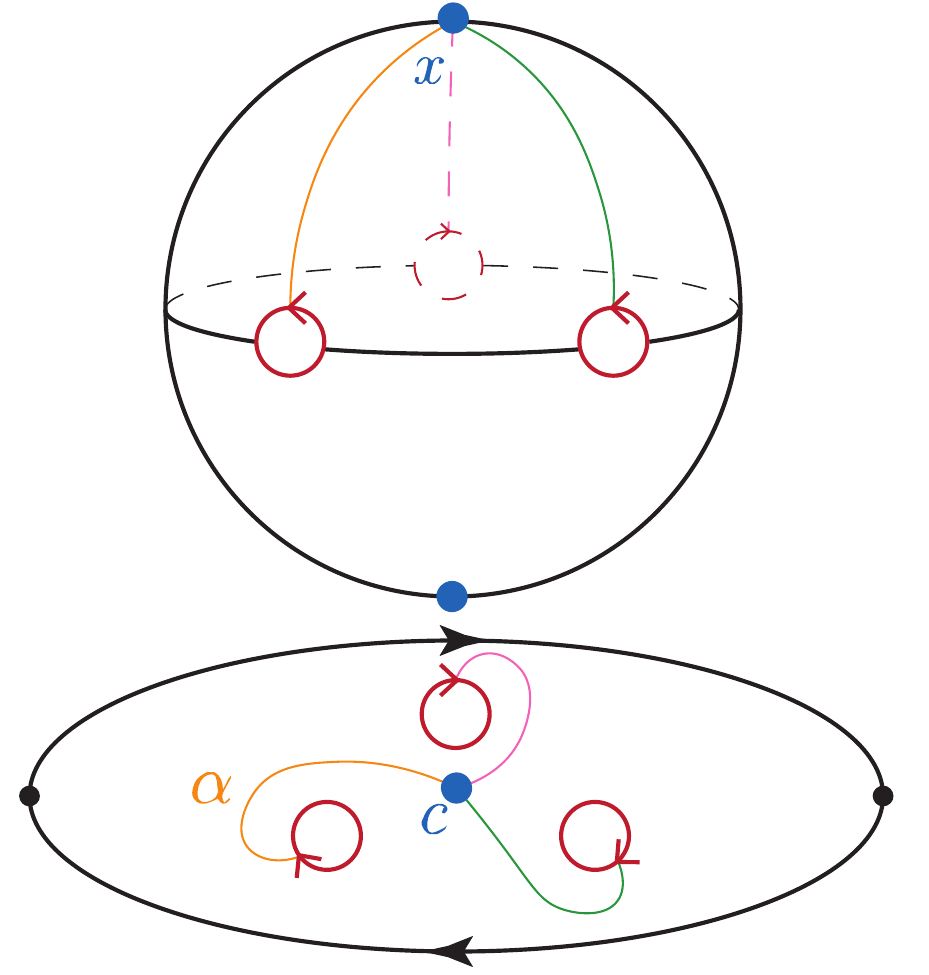}
\end{center}
\end{minipage} \ \begin{minipage}{0.4\textwidth}
\begin{center}
\includegraphics[scale=.3]{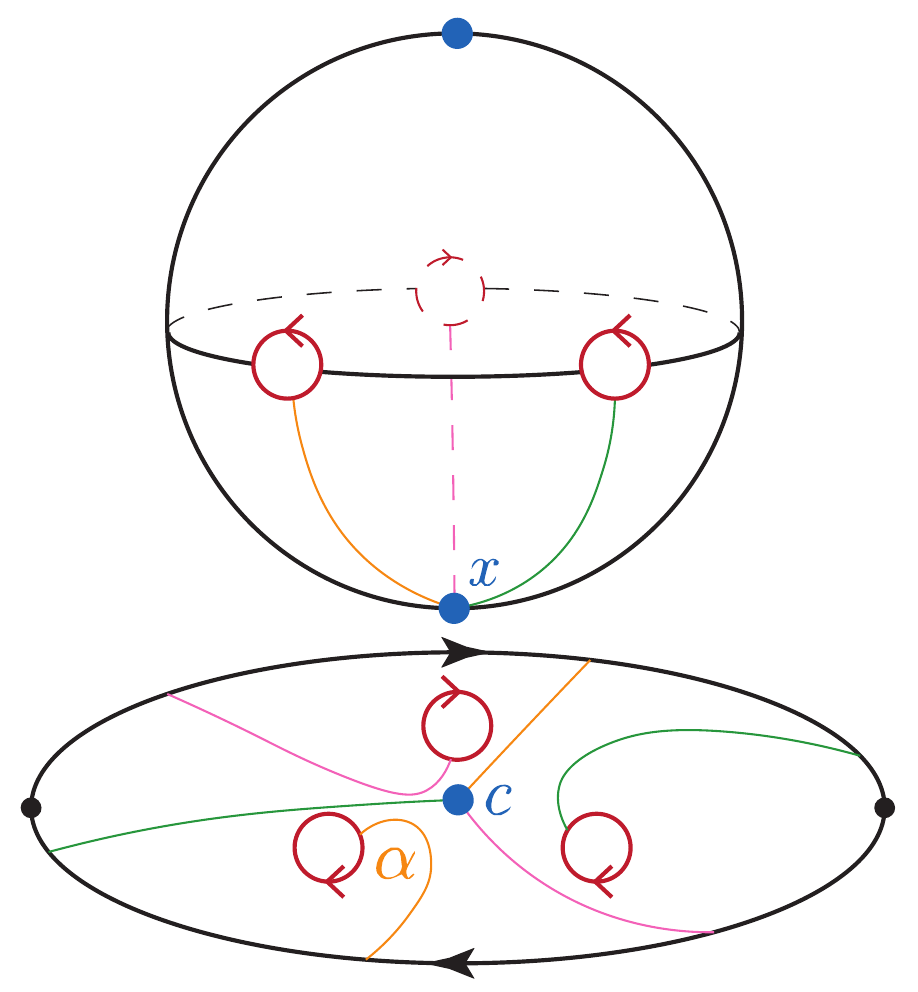}
\end{center}
\end{minipage}
\caption{\label{n5tat} Choices of $\alpha$ whose conjugates have a neighborhood isomorphic to $TR_3$.}
\end{figure}

\begin{figure}
\begin{minipage}{0.45\textwidth}
\begin{center}
\includegraphics[scale=.3]{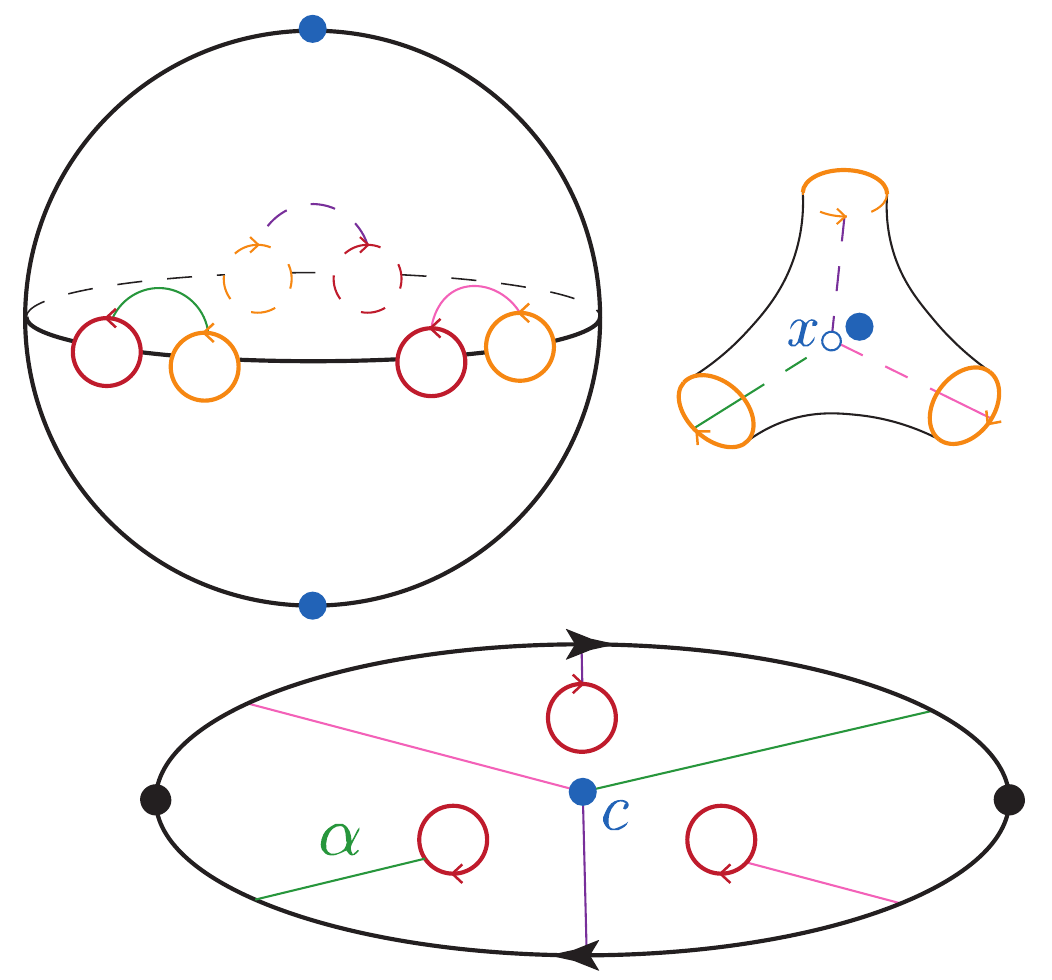}
\end{center}
\end{minipage}\ \begin{minipage}{0.45\textwidth}
\begin{center}
\includegraphics[scale=.3]{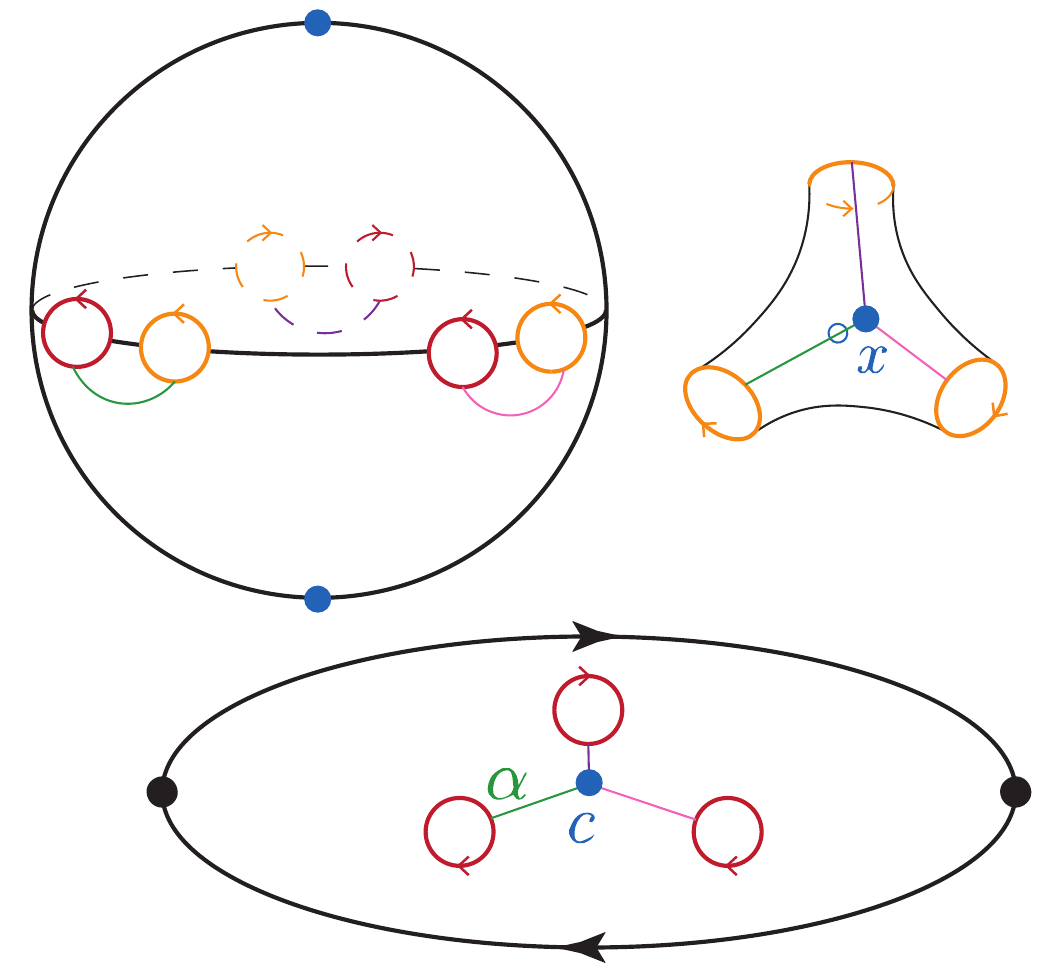}
\end{center}
\end{minipage}
\caption{\label{n9tat} Choices of $\alpha$ whose conjugates have a neighborhood isomorphic to $TR_3$.}
\end{figure}
\end{proof}

\begin{corollary}
If $X$ and $Y$ are closed, connected, non-orientable $C_p$-surfaces with $X-[TR_p]\cong Y-[TR_p]$, then $X\cong Y$. In particular, $X+_x[TR_p]$ is independent of the choice of $x$.
\end{corollary}

\bibliographystyle{amsalpha}
\bibliography{main}

\providecommand{\bysame}{\leavevmode\hbox to3em{\hrulefill}\thinspace}
\providecommand{\MR}{\relax\ifhmode\unskip\space\fi MR }
\providecommand{\MRhref}[2]{%
  \href{http://www.ams.org/mathscinet-getitem?mr=#1}{#2}
}
\providecommand{\href}[2]{#2}
\begin{thebibliography}{BCNS92}

\bibitem[AB67]{AB67}
M.~F. Atiyah and R.~Bott, \emph{A {L}efschetz fixed point formula for elliptic
  complexes. {I}}, Ann. of Math. (2) \textbf{86} (1967), 374--407. \MR{212836}

\bibitem[Aso76]{Asoh}
Tohl Asoh, \emph{Classification of free involutions on surfaces}, Hiroshima
  Math. J. \textbf{6} (1976), no.~1, 171--181. \MR{397755}

\bibitem[BCNS92]{BCN}
E.~Bujalance, A.~F. Costa, S.~M. Natanzon, and D.~Singerman, \emph{Involutions
  of compact {K}lein surfaces}, Math. Z. \textbf{211} (1992), no.~3, 461--478.
  \MR{1190222}

\bibitem[Bro91]{Bro91}
S.~Allen Broughton, \emph{Classifying finite group actions on surfaces of low
  genus}, J. Pure Appl. Algebra \textbf{69} (1991), no.~3, 233--270.
  \MR{1090743}

\bibitem[Chi69]{Chillingworth1969AFS}
D.~R.~J. Chillingworth, \emph{A finite set of generators for the homeotopy
  group of a non-orientable surface}, Proc. Cambridge Philos. Soc. \textbf{65}
  (1969), 409--430. \MR{235583}

\bibitem[Din97]{Ding}
Hongyu Ding, \emph{Classification of cyclic group actions on noncompact
  surfaces}, Pacific J. Math. \textbf{179} (1997), no.~2, 325--341.
  \MR{1452537}

\bibitem[Dug19]{Dug19}
Daniel Dugger, \emph{Involutions on surfaces}, J. Homotopy Relat. Struct.
  \textbf{14} (2019), no.~4, 919--992. \MR{4025595}

\bibitem[Haz20]{Haz19a}
Christy Hazel, \emph{The {$RO(C_2)$}-graded cohomology of {$C_2$}-surfaces in
  {$\underline{\mathbb{Z}/2}$}-coefficients}, Mathematische Zeitschrift
  \textbf{297} (2020).

\bibitem[Lic63]{Lickorish}
W.~B.~R. Lickorish, \emph{Homeomorphisms of non-orientable two-manifolds},
  Proc. Cambridge Philos. Soc. \textbf{59} (1963), 307--317. \MR{145498}

\bibitem[MP85]{MP85}
John~D. McCarthy and Ulrich Pinkall, \emph{Representing homology automorphisms
  of orientable surfaces}, Max Planck Inst. preprint MPI/SFB 85-11 (1985).

\bibitem[MP04]{MP04}
\bysame, \emph{Representing homology automorphisms of orientable surfaces},
  Available at
  \texttt{https://users.math.msu.edu/users/mccarthy/publications/repr.hom.auts.pdf}
  Accessed on July 13, 2023.

\bibitem[Nie37]{Nielsen}
Jakob Nielsen, \emph{Die {Struktur} periodischer {Transformationen} von
  {Fl{\"a}chen}}, Math.-{Fys}. {Medd}., {Danske} {Vid}. {Selsk}. 15, {No}. 1,
  1-77 (1937)., 1937.

\bibitem[Poh22]{Pohl}
Kelly Pohland, \emph{Ro({C}3)-{G}raded {B}redon {C}ohomology and
  {C}p-{S}urfaces}, ProQuest LLC, Ann Arbor, MI, 2022, Thesis
  (Ph.D.)--University of Oregon. \MR{4478855}

\bibitem[Sch29]{Sch29}
Willy Scherrer, \emph{Zur theorie der endlichen gruppen topologischer
  abbildungen von geschlossenen fl{\"a}chen in sich}, Commentarii Mathematici
  Helvetici \textbf{1} (1929), 69--119.

\bibitem[Smi67]{Smith}
P.~A. Smith, \emph{Abelian actions on {$2$}-manifolds}, Michigan Math. J.
  \textbf{14} (1967), 257--275. \MR{229236}

\bibitem[Sze06]{SzepMCG}
B{\l}a\.{z}ej Szepietowski, \emph{The mapping class group of a nonorientable
  surface is generated by three elements and by four involutions}, Geom.
  Dedicata \textbf{117} (2006), 1--9. \MR{2231155}

\end{thebibliography}

\end{document}